\documentclass[hidelinks,11pt]{article}

\newcommand{\bol}{\boldsymbol}

\newcommand{\ney}{\boldsymbol{y}}                          
\newcommand{\nex}{\boldsymbol{x}}

\newcommand{\ner}{\bol{r}}

\newcommand{\de}{\,\mathrm{d}}                               
\newcommand{\e}{\operatorname{e}}

\newcommand{\andtext}{\quad\mbox{and}\quad}

\newcommand{\p}{\partial}

\newcommand{\real}{\mathrm{Re}\,}    
                                   
\newcommand{\imag}{\mathrm{Im}\,}

\newcommand{\lf}{\left}
\newcommand{\rg}{\right}

\newcommand{\R}{\mathbb{R}}       
\newcommand{\C}{\mathbb{C}}

\usepackage{amsmath}
\DeclareMathOperator*{\argminA}{arg\,min} 

\usepackage{multirow}
\usepackage{amsmath}
\usepackage{amsfonts}
\usepackage{amsmath}
\usepackage{amssymb}  
\usepackage{graphicx}
\usepackage{caption}
\usepackage{mathrsfs}
\usepackage{upgreek}
\usepackage{amsthm}
\usepackage{subfig}
\usepackage{booktabs}
\usepackage{authblk}
\usepackage{xcolor}
\usepackage{cite}

\usepackage[]{algorithm2e}

\newtheorem{theorem}{Theorem}[section]
\newtheorem{lemma}[theorem]{Lemma}

\newtheorem{remark}[theorem]{Remark}

 \usepackage{xcolor}
 \usepackage[normalem]{ulem}

\title{On the regularization of Cauchy-type integral operators via \\the density interpolation method and applications
 }
\author[1,2]{Vicente G\'omez\thanks{vgh225@nyu.edu}}
\author[1]{Carlos P\'erez-Arancibia\thanks{cperez@mat.uc.cl}\thanks{This work was supported by FONDECYT under Grant 11181032.}}
\affil[1]{\small{Institute for Mathematical and Computational Engineering, Pontificia Universidad Cat\'olica de Chile}}
\affil[2]{\small{Courant Institute of Mathematical Sciences, New York University}}
\date{\today}

\begin{document}
\maketitle

\begin{abstract}

This paper presents a regularization technique for the high order efficient numerical evaluation of nearly singular,  principal-value, and finite-part  Cauchy-type integral operators. By relying on the Cauchy formula, the Cauchy-Goursat theorem, and on-curve Taylor interpolations of the input density, the proposed methodology allows to recast the Cauchy and associated integral operators as smooth contour integrals. As such, they can be accurately evaluated everywhere in the complex plane---including at problematic points near and on the contour---by means of elementary quadrature rules. Applications of the technique to the evaluation of the Laplace layer potentials and related integral operators, as well as to the computation conformal mappings, are examined in detail. The former application, in particular, amounts to a significant improvement over the recently introduced harmonic density interpolation method. Spectrally accurate discretization approaches for smooth and piecewise smooth contours are presented. A variety of numerical examples, including the solution of weakly singular and hypersingular Laplace boundary integral equations, and the evaluation of challenging conformal mappings, demonstrate the effectiveness and accuracy of the density interpolation method in this context. 
\end{abstract}

\section{Introduction} This contribution deals with the numerical evaluation of the  Cauchy integral operator
\begin{equation}\label{eq:int_op}
 (\mathcal C\varphi)(z):=\frac{1}{2\pi i}\int_\Gamma\frac{\varphi(\zeta)}{\zeta-z}\de \zeta,\quad z\in\C\setminus\Gamma,
\end{equation}
and related complex contour integrals, where $\Gamma$ is a closed curve enclosing a simply connected domain $\Omega\subset\C$ and $\varphi\in C(\Gamma)$, with the contour integral performed in the counterclockwise direction.  This class of integral expressions play a fundamental role in the analytical and numerical solution of numerous  problems in applied mathematics that relies on (complex) contour integral representations of smooth (analytic) functions, including, for instance,  conformal mapping~\cite{kythe2019handbook,henriciVol3,papamichael2010numerical} and boundary value problems in electrostatics, elastostatics, and potential and Stokes flows~\cite{jaswon1977integral,mikhlin1964integral,muskhelishvili2008} (see Section.~\ref{sec:applications}).

As is well known, the Cauchy integral operator~\eqref{eq:int_op} defines a complex analytic function $\mathcal C\varphi$  in both~$\Omega$ and in its unbounded open complement $\Omega'=\C\setminus\overline\Omega$~\cite{kress2012linear,henriciVol3}. Moreover, if $f$ is  an analytic function in $\Omega$, it holds that $\mathcal Cf=f$ in $\Omega$ and  $\mathcal Cf=0$ in $\Omega'$, or more conventionally,  
\begin{equation}\label{eq:CF}
 \frac{1}{2\pi i}\int_\Gamma \frac{f(\zeta)}{\zeta-z}\de \zeta=\begin{cases}f(z),&z\in\Omega,\\0,&z\in\Omega',\end{cases}
\end{equation}
which summarizes both the Cauchy integral formula and the Cauchy-Goursat theorem~\cite{henriciVol1}. Assuming further that $\Gamma$ is of class $C^2$ and that $\varphi\in C^{0,\alpha}(\Gamma)$ (i.e., that $\varphi$ is a H\"older continuous complex-valued function with exponent $0<\alpha\leq 1$), we have---by the Sokhotski-Plemelj theorem~\cite{kress2012linear,henriciVol3}---that the analytic function $\mathcal C\varphi$ in $\C\setminus\Gamma$ can be uniformly H\"older continuously extended from $\Omega$ to $\overline\Omega$ and from $\Omega'$ to  $\overline{\Omega'}$ with limiting values 
\begin{equation}\label{eq:sing_int}
\lim_{h\to 0}(\mathcal C\varphi)(z\pm h\nu(z))=\frac{1}{2}(H\varphi)(z) \mp \frac{1}{2}\varphi(z),\quad z\in\Gamma,
\end{equation}
where $\nu$ denotes the exterior unit normal to the contour $\Gamma$ and where $H\varphi$ is the principal-value integral
\begin{equation}\label{eq:hilbert}
(H\varphi)(z) :=\frac{1}{\pi i}\,{\rm p.v.}\!\!\int_\Gamma\frac{\varphi(\zeta)}{\zeta-z}\de \zeta,\quad z\in\Gamma.
\end{equation}
This paper introduces a unified approach for regularizing the contour integral expressions in~\eqref{eq:int_op} and~\eqref{eq:hilbert}, together with their corresponding derivatives, that enables their accurate numerical evaluation through direct use of elementary quadrature rules.

\begin{figure}[ht]
\centering	
\subfloat[Without regularization.]{\includegraphics[height=0.45\textwidth]{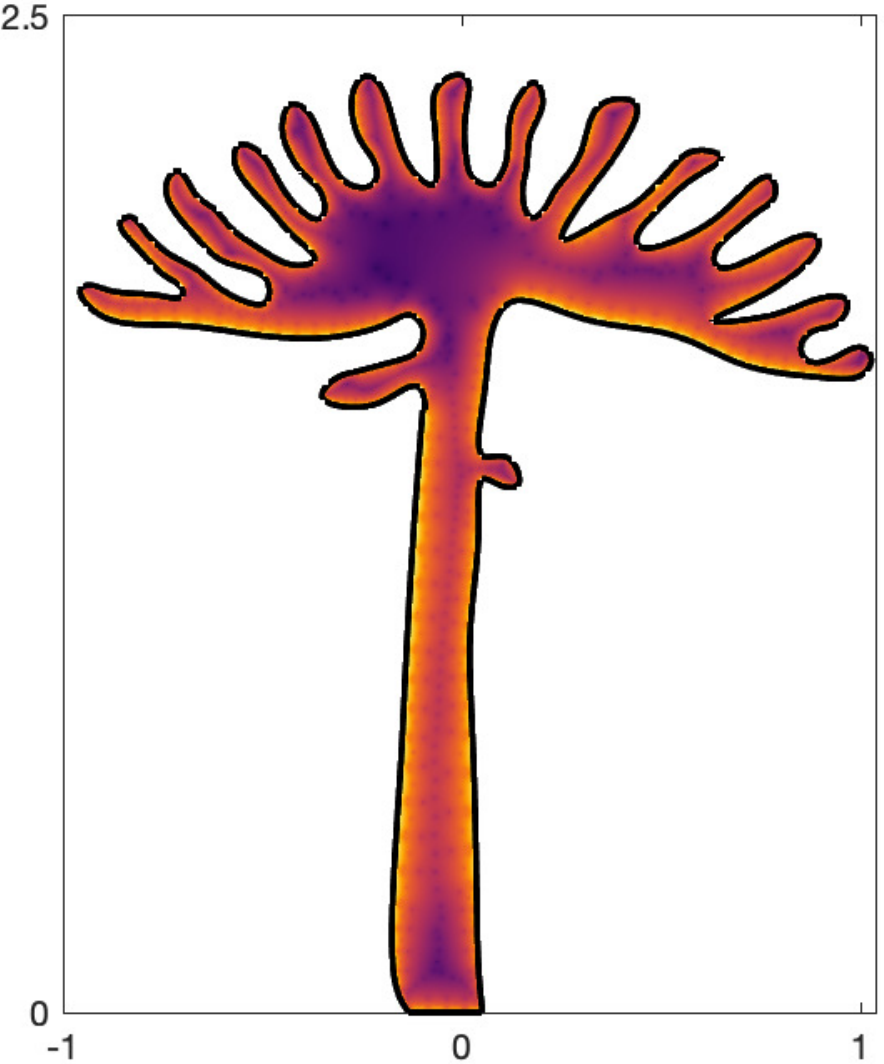}\label{fig_ANR}}\qquad
\subfloat[With regularization.]{\includegraphics[height=0.45\textwidth]{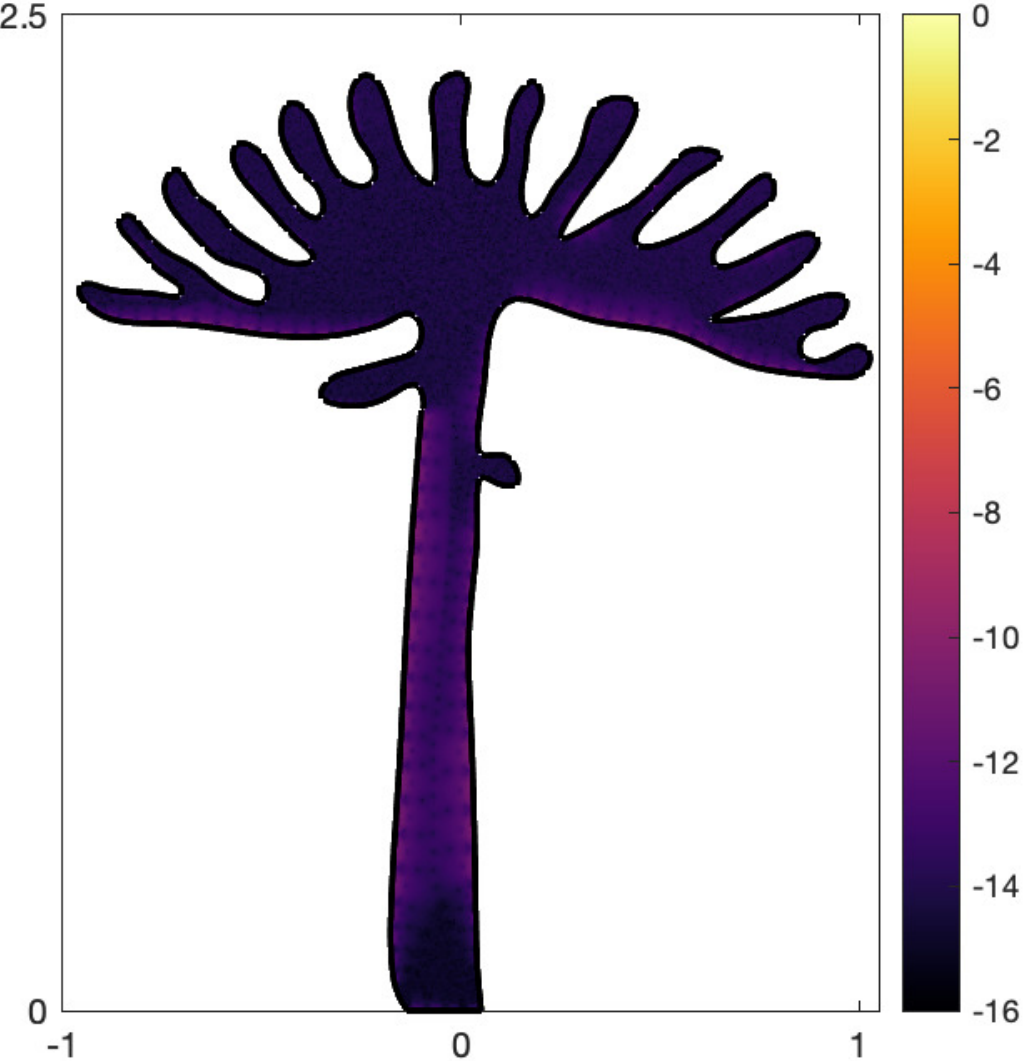}\label{fig_AR}}
 \caption{Logarithm in base ten of the absolute error in the numerical evaluation of the Cauchy operator $(\mathcal C\varphi)(z)$, defined in~\eqref{eq:int_op}, within an Araucaria tree-shaped contour parametrized by means of cubic splines. (a) Direct evaluation produced without using regularization, and  (b) using the proposed density interpolation method of order three, everywhere inside the curve. The input density function $\varphi(z)=\sin z$ and the same Chebyshev quadrature nodes were used in both examples.}\label{fig:AraucariaExample}
\end{figure}  

Since the pole singularity in the Cauchy integral~\eqref{eq:int_op} does not lie on~$\Gamma$, the contour integrand is as smooth as the input function~$\varphi$ and the contour~$\Gamma$. The numerical evaluation of~\eqref{eq:int_op} at any $z\in\C\setminus\Gamma$ could be accomplished, in principle, employing elementary quadrature rules. The trapezoidal rule, for example, yields exponential convergence as the number of quadrature nodes increases whenever both~$\varphi$ and~$\Gamma$ are analytic~\cite{lyness1967numerical,fornaro1973numerical}. Issues arise, however, when~$z$ approaches the contour, resulting in a nearly singular integrand (which is the term coined to refer to functions that although smooth develop large derivatives at certain points due to the presence of nearby singularities) at the points on $\Gamma$ that are the closest to $z$. At the numerical level, this phenomenon translates into a severe deterioration in the accuracy of the approximate integral, as the given set of quadrature nodes becomes incapable of properly resolving the localized features of the integrand taking place at the location of the nearly singular points, which could lie anywhere on $\Gamma$ depending on the location of~$z$. This problem is amplified when derivatives of the Cauchy operator are considered.  In order to illustrate the severity of the accuracy deterioration near the contour, we present Figure~\ref{fig:AraucariaExample} which displays the absolute error in the direct evaluation evaluation the Cauchy operator.

A number of approaches have been developed to tackle this problem. Lyness~\&~Delves~\cite{lyness1967numerical} developed a method that combines the Cauchy integral formula with Taylor series expansions at suitably located points inside the region enclosed by the contour. Instead of directly evaluating the integral~\eqref{eq:int_op}, this method produces the Taylor series expansion  $(\mathcal C\varphi)(z)\approx \sum_{j=0}^Nc_n(z-z_0)^n$ at a point $z_0\in\Omega$ sufficiently far from~$\Gamma$, by recasting the coefficients, via the Cauchy formula, as $c_n=n!(2\pi i)^{-1}\int_{\Gamma}\varphi(\zeta)/(\zeta-z_0)^{n+1}\de\zeta$. Making sure $z$ lies in inside the region of convergence of the series, one obtains an approximation of $(\mathcal C\varphi)(z)$. A similar but more direct and efficient approach was later introduced by Ioakimidis, Papadakis~\&~Perdios~\cite{ioakimidis1991numerical} which makes use of the Cauchy-Goursat theorem, instead of the Cauchy formula, and the trapezoidal rule. Some of the ideas put forth in these works have resurged in recent years in the form of  highly sophisticated and accurate algorithms for close surface evaluation of two-dimensional Laplace, Helmholtz and Stokes layer potentials operators~\cite{helsing2008evaluation,Barnett:2014tq,Barnett:2015kg} which suffer from the same nearly singular integrand problem. 

A different set of techniques have been developed for the numerical evaluation of challenging principal-value and finite-part contour integrals---such as~\eqref{eq:hilbert} and its tangential derivative considered below in Section~\ref{sec:DL_formulation}. There is ample literature on this particular subject (e.g., \cite{monegato1994numerical,paget1972algorithm,theocaris1979method,hunter1972some,chawla1974modified,chawla1974numerical,elliott1979gauss,Kolm:2001bt,monegato1982numerical}) that we do not attempt to review here. We do, however, mention a few important contributions concerning the evaluation of finite-part (hypersingular) contour integrals arising in the context of boundary integral equations, which include the quadrature by expansion method~\cite{klockner2013quadrature}, which bears similarities to~\cite{ioakimidis1991numerical} and~\cite{Barnett:2014tq}, and the spectrally accurate techniques based on trigonometric interpolation~\cite{Kress:1995} and trigonometric differentiation~\cite{kress2014collocation}.

This paper introduces a unified approach to the regularization of nearly-singular, principal-value, finite-part complex contour integrals. The proposed methodology relies on the ideas of density interpolation methods for boundary integral operators~\cite{perez2018plane,perez2019harmonic,perez2019planewave,perez2020IEEE}. It combines a Taylor-like interpolation of the contour density $\varphi$ at the nearly-singular (resp. singular) point on the contour $\Gamma$, with the Cauchy integral formula~\eqref{eq:CF} (resp. Sokhotski-Plemelj formula~\eqref{eq:sing_int}) to recast nearly-singular (resp. principal-value and finite-part) contour integrals in terms of integrands whose smoothness is controlled by the density interpolation order (see Section~\ref{sec:nearly_singular}).  The resulting contour integrals can thus be directly evaluated using elementary quadrature rules. Unlike the methods put forth in~\cite{lyness1967numerical,ioakimidis1991numerical}, which rely on~$\varphi$ being the restriction to $\Gamma$ of an analytic function in $\overline\Omega$, the density interpolation technique only requires smoothness of $\varphi$ at the interpolation points on the contour~$\Gamma$. On the other hand, the density interpolation approach bears a number of advantages in comparison with existing methods as it requires fewer number of parameters to be tuned in order to achieve its optimal performance.  It is worth mentioning that, as the recently introduced general-purpose density interpolation method~\cite{faria2020general}, the proposed technique is fully compatible with standard fast algorithms such as  $\mathcal H$-matrices~\cite{hackbusch2015hierarchical} and the Fast Multipole Method~\cite{greengard1987fast},  for the efficient evaluation of the integral operators hereby considered.

Two relevant applications of the proposed technique are discussed in detail in Section~\ref{sec:applications}, which concern (a) the regularization of the Laplace layer potentials and the associated boundary integral operators of Calder\'on calculus, and (b) the numerical evaluation of conformal mappings. Regarding the Laplace integral operators, this new density interpolation method amounts to a significant improvement over the related 2D~harmonic density interpolation method (HDI)~\cite{perez2019harmonic}. In particular, it allows for stand-alone kernel regularizations meaning that, unlike the HDI, where regularizations are effected by evaluation of pairs of integral operators, the present approach requires evaluation of just one integral operator to achieve the same integrand regularity degree. Regarding conformal mapping, on the other hand, it exploits the relation between conformal mappings and Laplace Dirichlet boundary value problems~\cite{symm1966integral,symm1967numerical} to derive Fredholm second-kind integral equations for the construction of both interior and exterior conformal mappings based on a Cauchy-operator integral representation of the double-layer potential. Upon regularization, the resulting conformal mappings can be accurately evaluated near and at the contour. 

High-order numerical methods for the practical implementation of the contour-integral regularization strategy are presented in Section~\ref{sec:numerics} for both smooth and piecewise smooth curves based on the trapezoidal and Fej\'er quadrature rules, respectively.  An efficient high-order FFT-based algorithm is developed for the construction of the Cauchy-operator density interpolant in Section~\ref{sec:det_coef}. Section~\ref{sec:examples}, finally, presents a variety of numerical examples designed to validate and demonstrate the effectiveness, applicability, and accuracy, of the proposed methodology.

\section{Regularization via density interpolation}\label{sec:nearly_singular}

We start off this section  by  addressing the problem of the regularization of the Cauchy integral operator~\eqref{eq:int_op} and its derivatives at points $z\in\C\setminus\Gamma$ near the contour~$\Gamma$. Assuming~$\Gamma$ to be a Jordan curve of class~$C^1$ and  provided the evaluation point $z$ in the Cauchy operator~\eqref{eq:int_op} lies close enough to~$\Gamma$, there is a unique~$z_0\in\Gamma$ such that 
\begin{equation}\label{eq:ns_point}
z_0= \argminA_{\zeta\in\Gamma}|z-\zeta|.\end{equation}
Our goal is then to regularize the integrand $g(\zeta)=\varphi(\zeta)/(\zeta-z)$ in~\eqref{eq:int_op} at and around the nearly-singular point $z_0\in\Gamma$, at which $g$ itself and also its derivatives reach large values. 

 In order to accomplish that we introduce the following Taylor-like complex polynomial
\begin{equation}\label{eq:truc_taylor}
P_N(z, z_0) :=  \sum_{j=0}^N \frac{c_j(z_0)}{j!}(z- z_0)^j,\quad z\in\C,\  z_0\in\Gamma,
\end{equation} where the set of coefficients $\{c_j(z_0)\}_{j=0}^N$ are to be determined by imposing appropriate interpolation conditions at $z_0$. Since $P_N(\cdot,z_0)$ is an entire function, the Cauchy integral formula~\eqref{eq:CF} together with the Cauchy--Goursat theorem~\cite{henriciVol1} yield the identity
\begin{equation}\label{eq:correction_formula}
\frac{1}{2\pi i}\int_{\Gamma}\frac{P_N(\zeta,z_0)}{\zeta-z}\de \zeta=\begin{cases}P_N(z, z_0),& z\in\Omega,\\
0,&z\in\Omega'.\end{cases}
\end{equation} 
Subtracting~\eqref{eq:correction_formula} from~\eqref{eq:int_op} we obtain that the Cauchy operator~\eqref{eq:int_op} can be recast as
\begin{equation}\label{eq:reg_formula}
(\mathcal C\varphi)(z)=\frac{1}{2\pi i}\int_{\Gamma}\frac{\varphi(\zeta)-P_N(\zeta,z_0)}{\zeta-z}\de \zeta  + \mathbf{ 1}_{\Omega}(z)P_N(z,z_0),\quad z\in\Omega\setminus\Gamma,\ z_0\in\Gamma,
\end{equation}
where $\mathbf{1}_{\Omega}$ denotes the indicator function of the domain $\Omega$. The main idea of the proposed methodology lies then  in constructing $P_N(\cdot,z_0)$, or equivalently, finding coefficients $\{c_j(z_0)\}_{j=0}^{N}$, so that the numerator of the integrand in~\eqref{eq:reg_formula} vanishes to high order precisely at $z_0$. In detail, we want 
\begin{equation}\label{eq:proto_interp_cond}
|\varphi(\zeta)-P_N(\zeta,z_0)|= o\lf(|\zeta-z_0|^{N}\rg)\quad\mbox{as}\quad \Gamma\ni\zeta\to z_0\in\Gamma,
\end{equation}  so that the regularized integrand in~\eqref{eq:reg_formula} satisfies
$$
\lf|\frac{\varphi(\zeta)-P_N(\zeta,z_0)}{\zeta-z}\rg|=  o\lf(\frac{|\zeta-z_0|^{{N}}}{\delta}\rg)\quad\mbox{as}\quad \Gamma\ni\zeta\to z_0\in\Gamma,
$$
where $\delta = |z-z_0|$. This implies that if both the curve $\Gamma$ and the density $\varphi$ are sufficiently smooth at $z_0$, not only the integrand in~\eqref{eq:proto_interp_cond} vanishes at $z_0$ but also all its derivatives up to oder $N\geq 1$, regardless of the distance $\delta$ from $z$ to the $\Gamma$. 

Assuming enough local regularity of the curve and the density at $z_0\in\Gamma$, the desired property~\eqref{eq:proto_interp_cond} can be recast in a more amenable form. Indeed,
 let $\gamma:[0,2\pi)\to\Gamma$ be a counterclockwise $2\pi$-periodic parametrization of $\Gamma$, and let  $\phi(\tau):=\varphi(\gamma(\tau))$ and $p_N(\tau,t_0):=P_N(\gamma(\tau),\gamma(t_0))$ for all $\tau\in [0,2\pi)$, with $z_0=\gamma(t_0)$ ($t_0\in [0,2\pi)$). 
Then, it can be directly shown---via Taylor series expansions---that~\eqref{eq:proto_interp_cond} is attained provided the interpolation conditions 
\begin{equation}\label{eq:inter_cond}
\lim_{\tau\to t_0}\frac{\p^m}{\p \tau^m}\lf[\phi(\tau)- p_N(\tau,t_0)\rg] = 0,\quad m=0,\ldots,N,
\end{equation}
are satisfied, which require both $\gamma$ and $\phi$ to be $N$~times differentiable at $\tau=t_0$.  

As it turns out, the interpolation conditions~\eqref{eq:inter_cond} suffice to uniquely determine the coefficients  $\{c_j(z_0)\}_{j=0}^{N}$. Indeed, differentiating $p_N(\cdot,t_0)$ at $t_0$, we obtain---by repeated use of the chain rule---the Fa\'a di Bruno formula
\begin{equation}\label{eq:chain_rule}
\left.\frac{\partial^{m}}{\partial \tau^{m}} p_N(\tau, t_0)\right|_{\tau=t_0}=\sum_{j=1}^{m} c_{j}(z_0) \mathbb{B}_{m, j}\left(\gamma^{\prime}(t_0), \gamma^{\prime \prime}(t_0), \ldots, \gamma^{(m-j+1)}(t_0)\right),\quad m=1,\ldots,N,
\end{equation}
  where $\mathbb{B}_{m, j}$ are the incomplete Bell polynomials~\cite{aldrovandi2001special}. From the interpolation condition~\eqref{eq:inter_cond} for $m=0$ we readily get that $c_0(z_0) = \phi(t_0)(=\varphi(z_0))$. Making use of~\eqref{eq:chain_rule}, on the other hand, we get from~\eqref{eq:inter_cond} that the~$N$~remaining coefficients are given by the solution of the linear system 
\begin{equation}\label{eq:linear_system}
A(t_0) \mathbf{c}(z_0)=\mathbf{b}(t_0),
\end{equation}
where $\mathbf{c}(z_0)=[c_1(z_0),\ldots,c_N(z_0)]^T\in\C^{N}$, $\mathbf{b}(t_0) =[\phi^{(1)}(t_0),\ldots, \phi^{(N)}(t_0)]^T\in\C^{N}$, and where $A(t_0)\in\C^{(N+1)\times (N+1)}$ is the lower triangular matrix with entries
$$a_{m, j}(t):=\begin{cases}0, & 1\leq m<j\leq N, \\ \mathbb{B}_{m, j}\left(\gamma^{\prime}(t), \ldots, \gamma^{(m-j+1)}(t)\right), & 1\leq j\leq m \leq N.\end{cases}$$ 
The existence and uniqueness of the coefficients $\{c_j(z_0)\}_{j=0}^{N}$ thus follows from the invertibility of~$A(z_0)$, which is a direct consequence of the fact that~$A(z_0)$ is a lower triangular matrix and its diagonal entries satisfy $a_{m,m}(t_0) =\mathbb B_{m,m}(\gamma'(t_0))=(\gamma'(t_0))^m\neq0$ ($\gamma$ is a regular parametrization of $\Gamma$).  It is worth mentioning here that in Section~\ref{sec:det_coef} we show that the coefficients $\{c_j(z_0)\}_{j=0}^{N}$ can be determined by an efficient recursive procedure so that the (somewhat laborious) construction of the matrix $A(t_0)$ can be completely avoided in practice.

The proposed approach can as well be utilized to the regularize the derivatives of the Cauchy operator~\eqref{eq:int_op}, which are given by:
\begin{equation}\label{eq:cauchy_der}
(\mathcal C\varphi)^{(n)}(z)=\frac{n!}{2\pi i}\int_\Gamma\frac{\varphi(\zeta)}{(\zeta-z)^{n+1}}\de \zeta,\quad z\in\Omega\setminus\Gamma,\ n\geq 1.
\end{equation}
In fact, using the Cauchy integral representation of the derivatives of the (analytic) density interpolant~\eqref{eq:truc_taylor}, we obtain  
\begin{equation}\label{eq:correction_formula_v2}
\frac{n!}{2\pi i}\int_{\Gamma}\frac{P_N(\zeta,z_0)}{(\zeta-z)^{n+1}}\de \zeta =\begin{cases}\displaystyle\frac{\p^n}{\p z^n}P_N(z, z_0) =\sum_{j=0}^{N-n} \frac{c_{j+n}(z_0)}{j!}(z- z_0)^j,& z\in\Omega,\\
0,&z\in\Omega',
\end{cases}
\end{equation}
for $N> n$.
Subtracting~\eqref{eq:correction_formula_v2} from~\eqref{eq:cauchy_der} we arrive at the regularized expression 
\begin{equation}\label{eq:reg_formula_der}
(\mathcal C\varphi)^{(n)}(z)=\frac{n!}{2\pi i}\int_{\Gamma}\frac{\varphi(\zeta)-P_N(\zeta,z_0)}{(\zeta-z)^{n+1}}\de \zeta  + \mathbf{ 1}_{\Omega}(z)\frac{\p^n}{\p z^n}P_N(z,z_0),\quad z\in\Omega\setminus\Gamma,
\end{equation}
for the $n$th-order derivative of the Cauchy operator, with the integrand satisfying 
\begin{equation}
\lf|\frac{\varphi(\zeta)-P_N(\zeta,z_0)}{(\zeta-z)^{n+1}}\rg|=  o\lf(\frac{|\zeta-z_0|^{N}}{\delta^{n+1}}\rg)\quad\mbox{as}\quad \Gamma\ni\zeta\to z_0\in\Gamma. \label{eq:to_please_the_reviewer}
\end{equation}

Finally, we apply the density interpolation technique to regularize  the Cauchy principal value integral~\eqref{eq:hilbert} for which we assume here that both $\Gamma$ and $\varphi$ are of class $C^N$, $N\geq 2$.
 Applying the Sokhotski-Plemelj formula~\eqref{eq:sing_int} to the analytic density interpolant~\eqref{eq:truc_taylor}, we obtain 
\begin{equation}\label{eq:cauchy_on_b}
P_N(z,z_0) =\frac{1}{\pi i}\,{\rm p.v.}\!\!\int_{\Gamma}\frac{P_N(\zeta,z_0)}{\zeta-z}\de \zeta
,\quad z\in\Gamma,
\end{equation}
where the limit was taken from inside of $\Omega$. Setting $z_0=z$ in the formula above and subtracting it from~\eqref{eq:hilbert} we arrive at 
\begin{equation}\label{eq:reg_hilbert}
(H\varphi)(z)=\frac{1}{\pi i}{\rm p.v.}\!\!\int_{\Gamma}\frac{\varphi(\zeta)-P_N(\zeta,z)}{\zeta-z}\de \zeta  + \varphi(z),\quad z
\in\Gamma,
\end{equation} where we have used that $P_N(z,z) = \varphi(z)$ by construction. The integrand in~\eqref{eq:reg_hilbert} is  smooth (in parametric form it is $(N-1)$ times differentiable at $\tau=t$, where $\zeta=\gamma(\tau)$ and $z=\gamma(t)$) and it satisfies 
$$
\lf|\frac{\varphi(\zeta)-P_N(\zeta,z)}{\zeta-z}\rg|=  o\lf(|\zeta-z|^{N-1}\rg)\quad\mbox{as}\quad  \Gamma\ni\zeta\to z_0\in\Gamma.
$$
It is easy to see that just the zeroth order interpolant (i.e., $\Phi_0(z,z_0) = \varphi(z_0)$) suffices to  produce a regular integrand in~\eqref{eq:reg_hilbert}. However, the real utility of the density interpolation method lies in that more singular (e.g., finite part) integrals associated with tangential derivatives of the principal value integral~\eqref{eq:hilbert} can be regularized by means of the proposed methodology, as is the case of the Laplace hypersingular operator addressed in Section~\ref{sec:DL_formulation}.

\begin{remark} Note that if $\varphi$ is the restriction to $\Gamma$ of a function $f$ that is (complex) analytic in a region containing $\Gamma$, we have that the expansion coefficients~\eqref{eq:truc_taylor} are given by  $c_j(z_0)=f^{(j)}(z_0)$, $j=0,\ldots,N$,  and hence $P_N$ is just the $N$th-order Taylor series expansion of $f$ at  $z_0\in\Gamma$. This is so because the analyticity   implies that the tangential derivatives of $\varphi$ along the curve $\Gamma$ coincide with the complex derivatives of $f$ in this case. This is clearly not necessarily true for less regular density functions~$\varphi$.
\end{remark}
\section{Applications}\label{sec:applications}
This section describes a variety of applications of the proposed regularization technique.

\subsection{Laplace boundary integral equations}\label{sec:DL_formulation}
The first application that we present concerns the solution of the Laplace equation by means of boundary integral equation methods. 

We start off by defining the double- and single-layer potentials in two spatial dimensions, which are given by 
\begin{subequations}\begin{eqnarray}
(\mathcal D\varphi)(\ner) &:=&\frac{1}{2\pi}\int_{\Gamma}\frac{\nu(\ney)\cdot(\ner-\ney)}{|\ner-\ney|^2}\varphi(\ney)\de s(\ney)\andtext\label{eq:DL_pot}\\
(\mathcal S\varphi)(\ner) &:=& -\frac{1}{2\pi}\int_{\Gamma} \log|\ner-\ney|\varphi(\ney)\de s(\ney),\quad \ner\in \R^2\setminus \Gamma,\label{eq:SL_pot}
\end{eqnarray}\label{eq:potentials}\end{subequations}
respectively, where  $\Omega\subset\R^2$ denotes a bounded and simply connected domain with boundary $\Gamma$ of class smooth $C^2$, and where the density function $\varphi:\Gamma\to\R$ is real-valued and continuous. These potentials define $C^2(\R^2\setminus\Gamma)$ functions that satisfy the Laplace equation in $\R^2\setminus\Gamma$. 

\begin{remark} Before continuing we warn the reader that in what follows of this section, the same symbol $\Omega$ is used to refer to both the (real) domain $\Omega\subset\R^2$ and its corresponding complex counterpart. Likewise, the same holds for the boundary $\Gamma$, its unit normal $\nu$, the density function $\varphi$, and the curve parametrization $\gamma$. We also mention that, in what follows we utilize the symbol $\ner$ to refer to points belonging to $\R^2\setminus\Gamma$ and the symbols~$\nex$ and~$\ney$ to denote points lying on the curve $\Gamma$.\end{remark}

As is well-known~\cite{kress2012linear}, the limit values of the double-layer potential~\eqref{eq:DL_pot} on $\Gamma$, give rise to the double-layer operator $K:C(\Gamma)\to C^{0,\alpha}(\Gamma)$:
\begin{equation}\label{eq:DL_op}\begin{split}
(K\varphi)(\nex) :=&\frac{1}{2\pi}\int_{\Gamma}\frac{\nu(\ney)\cdot(\nex-\ney)}{|\nex-\ney|^2}\varphi(\ney)\de s(\ney),\quad \nex\in\Gamma,\end{split}
\end{equation}
which is given in terms of a smooth (at least continuous) integral kernel~\cite{perez2019harmonic}. In view of the smoothness of the integrand in~\eqref{eq:DL_op}, numerical evaluation of the double-layer operator~\eqref{eq:DL_op} requires neither kernel regularization nor specialized quadrature rules. In contrast, the normal derivative of the double-layer potential on $\Gamma$ leads to  the so-called hypersingular operator $T:C^{1,\alpha}(\Gamma)\to C^{0,\alpha}(\Gamma)$:
\begin{equation}\label{eq:hypersingular}\begin{split}
(T\varphi)(\nex) =& \lim_{\epsilon\to 0}\nu(\nex)\cdot\nabla (\mathcal D\varphi)(\nex+\epsilon\nu(\nex))\\
=&\frac{1}{2\pi}\int_{\Gamma}\lf\{\frac{\nu(\ney)\cdot\nu(\nex)}{|\nex-\ney|^2}-2\frac{\nu(\ney)\cdot(\nex-\ney)(\nex-\ney)\cdot\nu(\nex)}{|\nex-\ney|^4}\rg\}\varphi(\ney)\de s(\ney),\quad \nex\in\Gamma,
\end{split}\end{equation} where the boundary integral in~\eqref{eq:hypersingular} has to be interpreted as a Hadamard finite-part integral. In view of the $O(|\nex-\ney|^{-2})$ asymptotic behavior of the kernel as $\Gamma\ni\ney\to\nex\in\Gamma$, numerical evaluation of the hypersingular operator~\eqref{eq:hypersingular} entails regularization via, e.g., integration by parts~\cite[Corollary 7.33]{kress2012linear}.

Similarly, the limit values of the single-layer potential~\eqref{eq:SL_pot} give rise to the single-layer operator $S:C(\Gamma)\to C^{0,\alpha}(\Gamma)$:
\begin{equation}\label{eq:SL_op}
(S\varphi)(\nex):= -\frac{1}{2\pi}\int_{\Gamma} \log|\nex-\ney|\varphi(\ney)\de s(\ney),\quad \nex\in\Gamma,
\end{equation}
which bears a weakly-singular kernel. Although integrable, the presence of the logarithmic singularity in~\eqref{eq:SL_op} makes the integrand not suitable for direct application of standard quadrature rules. The normal derivatives of the single-layer potential, on the other hand,  yields the adjoint double-layer operator $K^\top:C(\Gamma)\to C^{0,\alpha}(\Gamma)$:
\begin{equation}\label{eq:ADL_op}
(K^\top\varphi)(\nex) :=-\frac{1}{2\pi}\int_{\Gamma}\frac{\nu(\nex)\cdot(\nex-\ney)}{|\nex-\ney|^2}\varphi(\nex)\de s(\ney),\quad \nex\in\Gamma,
\end{equation}
which, as the double-layer operator, exhibits a smooth (at least continuous) kernel~\cite{perez2019harmonic}.

 
The next two theorems show that the proposed methodology can be applied to recast nearly-singular double- and single-layer potentials, as well as the hypersingular and the single-layer operators, in terms of smooth integrands of prescribed regularity. 

\begin{theorem}\label{th:double_layer}Let $\Omega\subset\R^2$ be a bounded simply connected domain with boundary $\Gamma=\{\gamma(t): t\in[0,2\pi)\}$ of class~$C^2$, and  let $\varphi\in C(\Gamma)$ be a real-valued function.  Assume that both $\gamma$ and~$\varphi\circ\gamma$ are $N$th-times continuously differentiable at $t_0\in[0,2\pi)$. Then,  the double-layer potential~\eqref{eq:DL_pot} can be expressed as: 
\begin{equation}
 (\mathcal D\varphi)(\ner)=-\real\lf\{\frac{1}{2\pi i}\int_{\Gamma}\frac{\varphi(\zeta)-P_N(\zeta,z_0)}{\zeta-z}\de \zeta  + \mathbf {1}_{\Omega}(z)P_N(z,z_0)\rg\},\label{eq:reg_dl} 
\end{equation}
for all $\ner=(\real z,\imag z)\in \R^2\setminus\Gamma$, where  $P_N(\cdot,z_0)$ is the $N$th-order  $\varphi$-interpolant at $z_0=\gamma(t_0)\in\Gamma$. 

Furthermore, assuming that $\gamma$ and $\varphi\circ\gamma$ are  $C^{N}([0,2\pi])$~functions, it holds that the hypersingular operator~\eqref{eq:hypersingular} can be expressed as:
\begin{equation}\label{eq:hyper_complex}
(T\varphi)(\nex) =-\real\lf\{\frac{\nu(z)}{2\pi i}\int_\Gamma\frac{\varphi(\zeta)-P_N(\zeta,z)}{(\zeta-z)^{2}}\de \zeta\rg\},
\end{equation}
for all $ \nex=(\real z,\imag z)\in\Gamma$.
\end{theorem}
\begin{proof}
We begin the proof by expressing the double-layer potential as~\cite{kress2012linear}
\begin{equation}\label{eq:corres_1}
 (\mathcal D\varphi)(\ner)=-\real\{(\mathcal C\varphi)(z)\} ,\quad \ner = (\real z,\imag z)\in\Omega\setminus\Gamma,
\end{equation} where the correspondence $\varphi(\ney) =\varphi(\zeta)$, $\ney=(\real \zeta,\imag \zeta)\in\Gamma$, between real and complex variables, has been used. Since $\gamma$ and $\varphi\circ \gamma$ are $N$~times differentiable at $t_0$, we have that $\varphi$ admits a density interpolant~$P_N$ at $z_0\in\Gamma$. Formula~\eqref{eq:reg_dl}  is hence directly obtained from the regularized expression for the Cauchy operator~\eqref{eq:reg_formula}.

Using the fact that the Cauchy operator~\eqref{eq:int_op} defines and analytic function in~$\C\setminus\Gamma$~\cite{kress2012linear}, on the other hand,  it follows from the Cauchy-Riemann equations that the gradient of the double-layer potential can be expressed as
\begin{equation}\label{eq:corres_2}
\nabla (\mathcal D\varphi)(\ner) = \lf(-\real\{{(\mathcal C\varphi)'(z)}\},\imag\{{(\mathcal C\varphi)'(z)}\}\rg),\quad\ner\in\Omega\setminus\Gamma.
\end{equation} 
Therefore, the hypersingular operator~\eqref{eq:hypersingular} can be expressed as
\begin{equation*}\label{eq:hyper_complex2}
( T\varphi)(\nex) =-\lim_{\epsilon\to 0}\real\{\nu(z)(\mathcal C\varphi)'(z+\epsilon\nu(z))\}=-\real\lf\{\frac{\nu(t)}{2\pi i}\,{\rm f.p.}\!\!\int_0^{2\pi}\frac{\phi(\tau)}{(\gamma(\tau)-\gamma(t))^{2}}\gamma'(\tau)\de \tau\rg\},
\end{equation*}
for all $\nex=(\real z,\imag z)\in\Gamma$, where $z =  \gamma(t)$, and $\phi = \varphi\circ \gamma$. The integral above can then be regularized by means of the identity
\begin{equation}\label{eq:sub}
\frac{1}{\gamma'(t)}\frac{\p}{\p t}{p_N(t, t_0)} =\frac{1}{\pi i}\,{\rm f.p.}\!\!\int_0^{2\pi}\frac{p_N(\tau,t_0)}{(\gamma(\tau)-\gamma(t))^2}\gamma'(\tau)\de \tau
,\quad t\in[0,2\pi),
\end{equation}
which is obtained by (carefully) differentiating $(H\varphi)(\gamma(t))$ in~\eqref{eq:cauchy_on_b} with respect to $t$.  Taking $t_0=t$ in~\eqref{eq:sub} and subtracting it from~\eqref{eq:hyper_complex} we finally arrive at 
\begin{equation*}\label{eq:hyper_complex_2}
(T\varphi)(\nex) =-\real\lf\{\frac{\nu(t)}{2\pi i}\int_0^{2\pi}\frac{\phi(\tau)-p_N(\tau,t_0)}{(\gamma(\tau)-\gamma(t))^{2}}\gamma'(\tau)\de \tau\rg\},\quad \nex=(\real\gamma(t),\imag\gamma(t))\in\Gamma,
\end{equation*}
which is obtained from: (a) 
$\lim_{\tau\to t}\p p_N(\tau,t)/\p \tau=\phi'(t) \in\R,$
that follows from the interpolation conditions~\eqref{eq:inter_cond} and that $\phi$ is a real valued function, and; (b)~$\real\{\nu(t)/\gamma'(t)\} = 0$ for all $t\in[0,2\pi)$. This concludes the proof.

\end{proof}

We now consider the single-layer potential and the single-layer operator. 

\begin{theorem}\label{th:single_layer}
Let $\Omega\subset\R^2$ be a bounded simply connected domain with boundary $\Gamma=\{\gamma(t): t\in[0,2\pi)\}$ of class~$C^2$, and let $\varphi\in C(\Gamma)$ be a real-valued function.  Let $\psi \in C(\Gamma)$ be defined in complex parametric form as  $\psi\circ \gamma=(\varphi\circ\gamma)\frac{|\gamma'|}{\gamma'}$, and assume that both~$\gamma$ and~$\psi\circ\gamma$ are $N$th-times continuously differentiable at $t_0\in[0,2\pi)$. Let also $\log(\cdot-z)$ be defined so that either its branch-cut path (see Figure~\ref{fig:branchCut}):
\begin{itemize}
\item Starts at $z\in\Omega$, exits $\Omega$ at  $z_0\in\Gamma$, and extends to infinity;
\item starts at $z\in\Omega'$ and extends to infinity without intersecting~$\Gamma$; or
\item starts at $z\in\Gamma$ and extends to infinity intersecting~$\Gamma$ only at $z$.
\end{itemize}Then, the single-layer potential~\eqref{eq:SL_pot} can be recast~as: 
\begin{equation}\label{eq:SL_pot_reg}
(\mathcal S\varphi)(\ner) = \imag\lf\{\frac{1}{2\pi i}\int_\Gamma \log(\zeta-z)\lf\{\psi(\zeta)-Q_N(\zeta,z_0)\rg\}\de \zeta- \mathbf{ 1}_{\Omega}(z)\int_{z_0}^z Q_N(\eta,z_0)\de\eta\rg\},
\end{equation}
for all $\ner=(\real z,\imag z)\in \R^2\setminus\Gamma$, where  $Q_N(\cdot,z_0)$ is the $N$th-order $\psi$-interpolant at $z_0=\gamma(t_0)\in\Gamma$.  Moreover, assuming that $\gamma$ and $\psi\circ\gamma$ are $C^{N}([0,2\pi])$ functions, it holds that the single-layer operator~\eqref{eq:SL_op} can be expressed as:
\begin{equation}\label{eq:SL_op_reg}
 (S\varphi)(\nex)=\imag\lf\{\frac{1}{2\pi i}\int_{\Gamma}\log(\zeta-z)\lf\{\psi(\zeta)-Q_N(\zeta,z)\rg\}\de \zeta  \rg\}, 
\end{equation}
for all $ \nex=(\real z,\imag z)\in\Gamma$. 
\end{theorem}
\begin{figure}[ht]
\centering	
\subfloat[$z\in\Omega.$]{\includegraphics[scale=1]{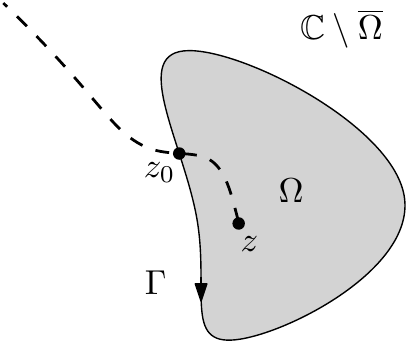}\label{fig_in}}\qquad
\subfloat[$z\in\Omega'=\C\setminus\overline\Omega.$]{\includegraphics[scale=1]{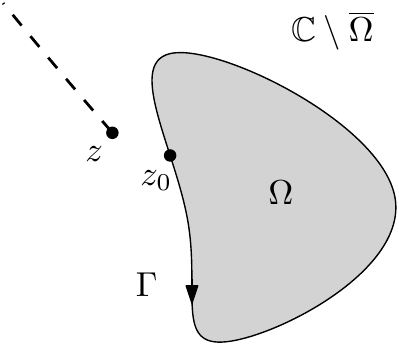}\label{fig_out}}\qquad
\subfloat[$z\in\Gamma.$]{\includegraphics[scale=1]{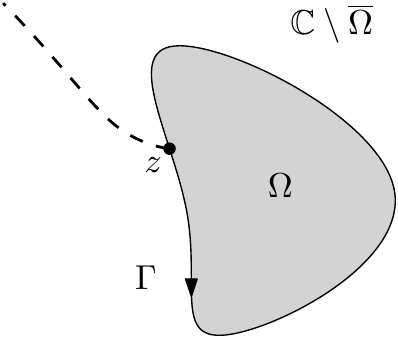}\label{fig_on}}
 \caption{Definition of the branch cut (dashed line) of $\log(\cdot-z)$ in Theorem~\ref{th:single_layer}, for the three relevant locations of the branch point~$z$.}\label{fig:branchCut}
\end{figure}  
\begin{proof}
First, we note that the single-layer potential~\eqref{eq:SL_pot} can be expressed as
\begin{equation}\label{eq:sl_0}
(\mathcal S\varphi)(\ner) = \imag\lf\{\frac{1}{2\pi i}\int_\Gamma \log(\zeta-z)\psi(\zeta)\de \zeta\rg\},\quad \ner = (\real z,\imag z)\in\R^2\setminus\Gamma,
\end{equation}
where $\psi\circ \gamma=(\varphi\circ\gamma)\frac{|\gamma'|}{\gamma'}$, independent of the selection of the branch of the logarithm. 

From the fact that $\gamma$ and $\psi\circ\gamma$ are $N$~times differentiable at $t_0$, it follows that $\psi$ admits a density interpolant $Q_N(\cdot,z_0)$ at $z_0=\gamma(t_0)\in\Gamma$, which is an entire function. For $z\in\Omega$ it holds,  by the Cauchy integral formula~\eqref{eq:CF}, that:
\begin{equation}\label{eq:sl_1}\begin{split}
\int_{z_0}^z Q_N(\eta,z_0)\de \eta=&\frac{1}{2\pi i}\int_\Gamma\lf(\int_{z_0}^z\frac{1}{\zeta-\eta}\de\eta\rg)Q_N(\zeta,z_0)\de\zeta=\\
&-\frac{1}{2\pi i}\int_\Gamma\lf\{\log(\zeta-z)-\log(\zeta-z_0)\rg\}Q_N(\zeta,z_0)\de\zeta,
\end{split}\end{equation}
where the branch of $\log(\cdot-z_0)$ is selected so that the resulting branch cut of the function inside the curly brackets is a simple open curve contained in $\Omega\cup\{z_0\}$, that connects $z\in\Omega$ and $z_0\in\Gamma$. Similarly, for $z\in\Omega'$ we have, by the Cauchy-Goursat theorem, that
\begin{equation}\label{eq:sl_2}\begin{split}
0=\frac{1}{2\pi i}\int_\Gamma\lf(\int_{z_0}^z\frac{1}{\zeta-\eta}\de\eta\rg)Q_N(\zeta,z_0)\de\zeta=-\frac{1}{2\pi i}\int_\Gamma\lf\{\log(\zeta-z)-\log(\zeta-z_0)\rg\}Q_N(\zeta,z_0)\de\zeta,
\end{split}\end{equation}
where the branch of $\log(\cdot-z_0)$ is chosen such that the branch cut of $\log(\cdot-z)-\log(\cdot-z_0)$ is a simple curve contained in $\Omega'\cup\{z_0\}$, that connects $z\in\Omega'$ and $z_0\in\Gamma$.
By construction then, in both cases the branch cut of $\log(\cdot-z_0)$ intersects  $\Gamma$ only at $z_0\in\Gamma$.

In order to simplify the formulae in~\eqref{eq:sl_1} and~\eqref{eq:sl_2} we consider the improper integrals:
\begin{equation*}
I_j=\int_\Gamma \log(\zeta-z_0)(\zeta-z_0)^{j}\de \zeta\quad\mbox{for } \quad j=0,\ldots,N.
\end{equation*}
Writing~$I_j$ in parametric form and using the $2\pi$-periodicity of $\gamma$,  we obtain
\begin{equation*}
I_j=\int_{t_0}^{t_0+2\pi} \log(\gamma(t)-z_0)(\gamma(t)-z_0)^{j}\gamma'(t)\de t=\lim_{\epsilon\to 0+}\lf.\frac{\log(\zeta-z_0)(\zeta-z_0)^{j+1}}{j+1}\rg|_{\zeta=\gamma(t_0-\epsilon)}^{\zeta=\gamma(t_0+\epsilon)}=0,
\end{equation*}
for all $j=0,\ldots,N$. Therefore, since the $\psi$-interpolant takes the form $Q_N(z,z_0)=\sum_{j=0}^N\frac{b_j(z_0)}{j!}(z-z_0)^{j}$, with coefficients $b_j(z_0)$, $j=0,\ldots,N$, depending on  $\psi\circ\gamma$ and its derivatives at $t_0$,  we conclude that 
\begin{equation}\label{eq:sl_3}
\frac{1}{2\pi i}\int_\Gamma\log(\zeta-z_0)Q_N(\zeta,z_0)\de\zeta =0,
\end{equation}
in both~\eqref{eq:sl_1} and~\eqref{eq:sl_2}.
%

Therefore, combining~\eqref{eq:sl_0},~\eqref{eq:sl_1},~\eqref{eq:sl_2} and \eqref{eq:sl_3}, we arrive at
\begin{equation}\label{eq:sl_reg_reg}\begin{split}
(\mathcal S\varphi)(\ner) =&~\imag\lf\{\frac{1}{2\pi i}\int_\Gamma \log(\zeta-z)\lf\{\psi(\zeta)-Q_N(\zeta,z_0)\rg\}\de\zeta-\mathbf{1}_{\Omega}(z)
 \int_{z_0}^zQ_N(\eta,z_0)\de\eta\rg\}
\end{split}\end{equation}
for all $\ner = (\real z,\imag z)\in\R^2\setminus\Gamma.$ 

Finally, taking the limit of~\eqref{eq:sl_reg_reg} as $\C\setminus\Gamma\ni z\to z_0\in\Gamma$, we obtain
$$
(S\varphi)(\nex)=\imag\lf\{\frac{1}{2\pi i}\int_\Gamma \log(\zeta-z_0)\lf\{\psi(\zeta)-Q_N(\zeta,z_0)\rg\}\de \zeta\rg\},
$$ for all $\nex=(\real z_0,\imag z_0)\in\Gamma$. This concludes the proof.
\end{proof}

\begin{remark}
We here discuss the construction of the function $\log(\cdot-z)$ in the definition of the single-layer potential, which has to fulfill the contour-dependent conditions stated in Theorem~\ref{th:single_layer} (see Figure~\ref{fig:branchCut}). In the case of a contour $\Gamma$ enclosing a star-shaped domain $\Omega$ with respect to $z^\star\in\Omega$, the branch cut of the logarithm $\log(\cdot -z)$ can be simply selected as the path stemming from $z$ and extending along the straight line~$\{w=z+s(z-z^\star)\in\C, s\geq 0\}$. 

For more intricate contours, however, defining a suitable branch cut of $\log(\cdot-z)$ could be a painstaking process. To help defining $\log(\cdot-z)$ in a systematic manner, it can be expressed as
\begin{equation}\label{eq:log_del}
\log(\cdot-z) =\operatorname{Log}\lf(\frac{\cdot-z}{\cdot-z_0}\rg)+\operatorname{Log}\lf(\frac{\cdot-z_0}{\cdot-z_1}\rg)+\cdots+\operatorname{Log}\lf(\frac{\cdot-z_{Q-1}}{\cdot-z_{Q}}\rg)+\log_{\tau}(\cdot-z_Q),
\end{equation}
where $\operatorname{Log}$ is the principal branch and where  $\log_{\tau}(\cdot-z_Q)$ has its branch cut along the straight line $\{w=z_Q+s\tau\in\C, s\geq 0\}$  $(0\neq\tau\in\C)$. The resulting brach-cut path of~\eqref{eq:log_del} is a piecewise linear curve starting at $z$, passing through the control points ${z_q}$, $q=0,\ldots,Q$, and then extending to infinity in the direction determined by $\tau$. Making use of a triangulation of an annular region $B_R\setminus\overline\Omega$, where  $B_R = \{z\in\C:|z-z^*|<R\}$ is a disk containing~$\Omega$, the control points can be selected as the centroids of elements forming a ``chain" consisting of triangles sharing a common segment.  The chain starts from the triangle containing $z$, if $z\in B_R\setminus\overline\Omega$, or from the triangle containing $z_0={\rm arg}\min_{\zeta\in\Gamma}|z-\zeta|$, if $z\in\overline\Omega$, and ends at a triangle intersecting the outer boundary $\p B_R$. The direction the line extending from $z_Q$ to infinity can be selected as $\tau = z_Q-z^*,$ where $z_Q$ is the centroid of the last triangle of the chain.

\end{remark}

\begin{remark} It follows from equations~\eqref{eq:corres_1} and~\eqref{eq:sl_0} that also the gradients of the double- and single-layer layer potentials can be regularized by means of the proposed methodology. Indeed, from~\eqref{eq:corres_2}  the gradient of the double-layer potential can be expressed in terms of the derivative of the Cauchy operator, which can in turn be recast as~\eqref{eq:reg_formula_der}. Similarly, for the gradient of the single-layer potential we have the identity
$$
\frac{\de }{\de z}\lf\{\frac{1}{2\pi i}\int_\Gamma \log(\zeta-z)\psi(\zeta)\de \zeta\rg\} = -\frac{1}{2\pi i}\int_\Gamma \frac{\psi(\zeta)}{\zeta-z}\de \zeta = -(\mathcal C\psi)(z),\quad z\in\C\setminus\Gamma,
$$
where, as in Theorem~\ref{th:single_layer}, $\psi\circ \gamma= (\varphi\circ\gamma) \frac{|\gamma'|}{\gamma'}$. Using the Cauchy-Riemann equations, this identity leads to
$$
\nabla (\mathcal S\varphi)(\ner) = -\lf(\imag (\mathcal C\psi)(z),\real (\mathcal C\psi)(z)\rg),\quad\ner = (\real z,\imag z)\in\R^2\setminus\Gamma,$$
where the Cauchy operator can be regularized by means of~\eqref{eq:reg_formula}.
\end{remark}

It is thus clear from the results in this section that, provided $\varphi$ and $\Gamma$ are smooth enough, the density interpolation technique can be applied to regularize singular and nearly singular integrals associated with the evaluation of the Laplace double- and single-layer potentials, their gradients, and all four integral operators of Calder\'on calculus.

\subsection{Conformal mapping }
Yet another important application of the proposed methodology concerns conformal mapping, for which we present a straightforward integral equation method based on the double-layer formulation of both interior and exterior Laplace Dirichlet problems. Some related works on this huge subject include the single-layer formulation put forth by G.~T.~Symm in~\cite{symm1966integral,symm1967numerical} and the recent contribution~\cite{wala2018conformal} which, as the one considered here, relies on a double-layer formulation. (We refer the interested reader to~\cite[Sec.~16.7]{henriciVol3} and to the more recent handbook~\cite{kythe2019handbook} for surveys on second-kind integral equations for conformal mapping applications.) 

\paragraph{Interior regions.}We first consider the problem of mapping a simply connected domain $\Omega\subset\C$, with Jordan boundary $\p\Omega=\Gamma$ of class $C^2$ enclosing the origin, conformally onto the unit disk $D=\{z\in\C:|z|<1\}$. 

Let~$f_i:\overline\Omega\to\overline D$ denote the unique conformal mapping satisfying $f_i(0) = 0$ and $f'_i(0)>0$~\cite[Corollary 5.10c]{henriciVol3}. Consider then the function $g_i(z) = \log(f_i(z)/z)$ which is analytic in $\Omega$ and satisfies $\real g_i(z) = -\log |z|$ for $z\in\Gamma$. Clearly, $f_i(z) = z\e^{g_i(z)}$ satisfies $f_i(0)=0$. Now, writing $g_i(z) = u_i(\ner) + iv_i(\ner)+i\alpha$, $\ner = (\real z,\imag z)$, $\alpha\in\R$, it follows from Cauchy-Riemann equations that~$u_i$ and~$v_i$ are harmonic conjugates of each other in $\Omega$.  Therefore, in particular, $u_i\in C^2(\Omega)\cap C(\overline\Omega)$ is the unique solution~\cite[Theorem 6.23]{kress2012linear} of the following Dirichlet interior boundary value problem:
$$
\Delta u_i=0\ \mbox{ in }\ \Omega,\quad u_i(\nex)= -\log|\nex|, \quad \nex\in\Gamma.
$$

Adopting the notation used in Section~\ref{sec:DL_formulation} and looking for a solution in the form of the double-layer potential:
$$
u_i(\ner) = (\mathcal D\varphi_i)(\ner)=-\real\{(\mathcal C\varphi)(z)\},\quad \ner=(\real z,\imag z)\in\Omega,
$$
we obtain that the unknown real-valued density $\varphi_i\in C(\Gamma)$ is the unique solution  of the second-kind integral equation~\cite[Theorems 6.21 and 6.22]{kress2012linear}
\begin{equation}\label{eq:int_IE}
-\frac{\varphi_i(\nex)}{2} + (K\varphi_i)(\nex)=-\log|\nex|,\quad\nex\in\Gamma,
\end{equation}
where the operator $K:C(\Gamma)\to C(\Gamma)$ is the double-layer operator~\eqref{eq:DL_op}. Upon solving~\eqref{eq:int_IE} for~$\varphi_i$ and using that 
$$v_i(\ner)=-\imag\{(\mathcal C\varphi_i)(z)\}, \quad\ner\in\Omega,$$
is a harmonic conjugate of $u_i$, we conclude that the sought conformal mapping is given by
\begin{equation}\label{eq:conf_map}
f_i(z) =z\exp\lf(-(\mathcal C\varphi_i)(z)-iv_i(\mathbf 0)\rg),\quad z\in\Omega.
\end{equation}
where the real constant $\alpha$ has been selected so that $f'_i(0)=\exp(u_i(\mathbf 0))>0$.

\paragraph{Exterior regions.}Consider now the problem of mapping $\Omega' =\C\setminus\overline\Omega$ conformally onto the exterior domain of the unit disk $D' = \{z\in\C:|z|>1\}$.  

As is well known~\cite[Theorem 5.10c]{henriciVol3}, there exists a unique mapping function $f_e:\overline{\Omega'}\to\overline{D'}$ such that $f_e(\infty) = \infty$ with Laurent series at infinity given by
$$
f_e(z) = \upgamma^{-1} z+a_0+a_1z^{-1}+\cdots,
$$ where $\upgamma>0$ is the so-called capacity of $\Gamma$.  As in the interior region case, looking for a conformal mapping of the form $f_e(z) = z\exp(g_e(z))$,t we have that $g_e(z) =u_e(\ner)+iv_e(\ner)$, $\ner=(\real z,\imag z)$, is analytic in $\Omega'$ and satisfies $\real g_e(z)=-\log|z|$ for $z\in\Gamma$. Therefore,  $u_e$ and $v_e$ are harmonic conjugate of each other and  $u_e\in C^2(\Omega')\cap C(\overline{\Omega'})$ is the unique bounded solution~\cite[Theorem 6.25]{kress2012linear} of the exterior Dirichlet problem:
$$
\Delta u_e=0\ \mbox{ in }\ \Omega',\quad u_e(\nex)=-\log|\nex|, \quad \nex\in\Gamma.
$$
Following~\cite{kress2012linear}, we look for the solution in the form of a modified double-layer potential:
\begin{equation}\label{eq:dl_pot_mod}
u_e(\ner) = (\mathcal D\varphi_e)(\ner)+\int_\Gamma \varphi_e(\ney)\de s(\ney)=-\real\{(\mathcal C\varphi)(z)\}+\int_\Gamma \varphi_e(\ney)\de s,\  \ner=(\real z,\imag z)\in\Omega'.
\end{equation}
Imposing the boundary condition on $\Gamma$, we thus arrive at the following uniquely solvable second-kind integral equation~\cite[Theorems 6.24 and 6.25]{kress2012linear} for the unknown density $\varphi_e\in C(\Gamma)$:
\begin{equation}\label{eq:ext_IE}
\frac{\varphi_e(\nex)}{2} + (K\varphi_e)(\nex)+\int_\Gamma \varphi_e(\ney)\de s(\ney)=-\log|\nex|,\quad\nex\in\Gamma.
\end{equation}
The unique conformal mapping satisfying the condition $\upgamma^{-1} = \lim_{z\to\infty}f_e(z)/z>0$ is thus given by
\begin{equation}\label{eq:ext_conf_map}
f_e(z)= z\exp\lf(-(\mathcal C\varphi_e)(z)+\int_\Gamma \varphi_e(\ney)\de s\rg),\quad z\in\Omega'.
\end{equation}

As we show in the numerical examples presented in Section~\ref{sec:examples}, the proposed methodology can be directly applied to produce accurate numerical evaluations of both interior~\eqref{eq:conf_map} and exterior~\eqref{eq:ext_conf_map} conformal mappings at points close to and on the contour $\Gamma$. 

%
%
%
%

\subsection{Biharmonic equation, Stokes flow, and elastostatics}
We briefly mention other linear elliptic boundary value problems whose solutions can be formulated in terms of Cauchy-like integral operators. These are Stokes flow~\cite{Greengard:2004dg,greengard1996integral,kropinski1999integral,Kropinski:2002fd} and elastostatic~\cite{greengard1996integral,helsing2001complex} problems in the plane, which relying on the complex variable theory for the biharmonic equation~\cite{greenbaum1992numerical,mayo1984fast}, make heavy use of Goursat potentials sought as
\begin{equation*}\begin{split}
\phi(z)=\frac{1}{2 \pi i} \int_{\Gamma} \frac{\omega(\zeta)}{\zeta-z} \mathrm{d} \zeta\quad\mbox{and}\quad
\psi(z)=\frac{1}{2 \pi i} \int_{\Gamma} \frac{\bar{\omega}(\zeta) \mathrm{d} \zeta+\omega(\zeta)\de\overline\zeta}{\zeta-z}-\frac{1}{2 \pi {i}} \int_{\Gamma} \frac{\bar{\zeta} \omega(\zeta)}{(\zeta-z)^{2}} \mathrm{d} \zeta.\end{split}\end{equation*} 

Clearly, these contour integrals can be directly regularized by means of  the proposed methodology at points $z\in\C\setminus\Gamma$ close to the contour $\Gamma$. For instance, writing the second integral in parametric form, we obtain
$$\int_{\Gamma} \frac{\bar{\omega}(\zeta) \mathrm{d} \zeta+\omega(\zeta)\de\overline\zeta}{\zeta-z} = \int_0^{2\pi}\frac{\phi(\tau)}{\gamma(\tau)-\gamma(t)}\gamma'(\tau)\de \tau,
$$
where $\phi(\tau) = \overline{w(\gamma(\tau))}+w(\gamma(\tau))\overline{\gamma'(\tau)}/\gamma'(\tau)$ with $\gamma:[0,2\pi)\to\Gamma$ being the contour parametrization. Regularization of this integral can be directly achieved by writing the resulting Cauchy operator as in~\eqref{eq:reg_formula} in terms of the density interpolant of $\varphi=\phi\circ\gamma^{-1}$.

\section{Numerics}\label{sec:numerics}
This section presents numerical procedures for the implementation of the proposed density interpolation technique in the case of smooth and piecewise smooth contours $\Gamma$. We focus here on the numerical evaluation of the regularized Cauchy integral~\eqref{eq:reg_formula} and its derivatives~\eqref{eq:reg_formula_der} which entail the construction of the $N$th-order density interpolant $P_N$ and evaluation of the contour integral
\begin{equation}\label{eq:model_int}
\int_\Gamma f(\zeta)\de \zeta\quad\mbox{with}\quad f(\zeta):= \frac{n!}{2\pi i}\frac{\varphi(z)-P_N(\zeta,z_0)}{(z-\zeta)^{n+1}},
\end{equation}
where  $z_0\in\Gamma$, $z\in\C\setminus\Gamma$, and  $N>n\geq 0$.

The same procedures can be used for the evaluation of the regularized Cauchy principal-value integral in~\eqref{eq:reg_hilbert}, the regularized finite-part integral in~\eqref{eq:hyper_complex_2}, and the regularized weakly singular integral in~\eqref{eq:SL_op_reg}, provided proper care is taken at the singular point; meaning that the corresponding integrand there is either set to zero or computed by means of a L'Hospital limit, depending on the interpolation order $N$ used.

\subsection{Smooth contours} \label{sec:smooth}
Consider first a smooth Jordan curve $\Gamma$ that admits a global smooth $2\pi$-periodic parametrization $\gamma:[0,2\pi)\to\C$. Applying the trapezoidal quadrature rule the regularized contour integral~\eqref{eq:model_int} can be directly approximated as 
\begin{subequations}\label{eq:trap_rule}
\begin{equation}\label{eq:trap}
\int_\Gamma f(\zeta)\de \zeta=\int_\Gamma f(\gamma(t))\gamma'(t)\de t\approx \frac{2\pi}{M}\sum_{m=1}^Mf(\gamma(t_m))\gamma'(t_m),
\end{equation}
where the quadrature nodes are given by  
\begin{equation}\label{eq:uniform}
t_m:=\frac{2\pi}{M}(m-1)\quad \mbox{for} \quad m=1,\ldots,M.\end{equation}\label{eq:trap_rule_total}\end{subequations}

Given $z\in\C\setminus\Gamma$, the interpolation point $z_0\in\Gamma$~\eqref{eq:ns_point} is approximated from the discrete set of contour points $\{\gamma(t_m)\}_{m=1}^{M}$, i.e., $z_0\approx z^*$ with  $$z^*=\gamma(t^*)\quad\mbox{where}\quad t^* =\underset{1\leq m\leq M}{\rm arg\,min}|z-\gamma(t_m)|.$$
 
As is well known~\cite{trefethen2014exponentially} the trapezoidal quadrature rule~\eqref{eq:trap} yields exponential convergence when applied to analytic integrands while it yields superalgebraic convergence for $C^\infty$ integrands~\cite{davis2007methods}.  Another significant advantage of using the trapezoidal rule in this context, is that the derivatives $\phi^{(n)}(t^*)=(\varphi\circ\gamma)^{(n)}(t^*)$ and $\gamma^{(n)}(t^*)$ for $n=1,\ldots N,$ at the interpolation/quadrature point $t^*$---which are required in the construction of the density interpolant $P_N$---can be accurately and efficiently computed from $\{\phi(t_m)\}_{m=1}^{M}$ and $\{\gamma(t_m)\}_{m=1}^{M}$ by means of FFT-based differentiation~\cite{johnson2011notes}. A more detailed algorithmic description of an efficient numerical procedure for constructing $P_N$ is presented in Section~\ref{sec:det_coef} below.

\subsection{Piecewise smooth contours} \label{sec:psmooth}
Consider now a piecewise smooth Jordan curve given by the union $\Gamma=\bigcup_{p=1}^P \Gamma_p$ of $P\geq1$ non-overlapping patches $\Gamma_p$, $p=1,\ldots, P$, of class $C^N$.  Letting  $\gamma_p:[-1,1]\to \Gamma_p$ denote  the parametrizations of the patches, the contour integral~\eqref{eq:model_int} is here approximated as
\begin{subequations}\label{eq:fejer_rule}\begin{equation}\label{eq:fejer}
\int_\Gamma f(\zeta)\de \zeta = \sum_{p=1}^P \int_{\Gamma_p}f(\zeta)\de \zeta =\sum_{p=1}^P\int_{-1}^1 f(\gamma_p(t))\gamma_p'(t)\de t\approx\sum_{p=1}^P\sum_{m=1}^M\omega_mf(\gamma_p(t_m))\gamma'_p(t_m),
\end{equation}
by applying the Fej\'er quadrature rule~\cite{davis2007methods} with nodes given by the Chebyshev zero points~\cite{vetterling1992numerical}:
\begin{equation}\label{eq:Cheb_points}
t_{m}:=\cos \left(\vartheta_{m}\right), \quad \vartheta_{m}:=\frac{(2 m-1) \pi}{2 M}, \quad m=1, \ldots, M,
\end{equation}
and  weights given by:
\begin{equation}\label{eq:Fweights}
\omega_{m}:=\frac{2}{M}\left(1-2 \sum_{\ell=1}^{[M / 2]} \frac{1}{4 \ell^{2}-1} \cos \left(2 \ell \vartheta_{m}\right)\right), \quad m=1, \ldots, M.\end{equation}\end{subequations}
The Fej\'er weights~\eqref{eq:Fweights} can be efficiently computed via the FFT~\cite{waldvogel2006fast}. As the trapezoidal quadrature rule~\eqref{eq:trap_rule}, the F\'ejer rule~\eqref{eq:fejer_rule} yields high-order accuracy for integration of smooth functions~\cite{davis2007methods}. Furthermore, the derivatives of the density and the parametrization at the interpolation point
$$z^*=\gamma_{p^*}(t^*)\quad\mbox{where}\quad t^* =\underset{1\leq m\leq M,1\leq p\leq P}{\rm arg\,min} |z-\gamma_p(t_m)|,$$ 
which are needed in the construction of the interpolant $P_N$, can be obtained from $\{\varphi(\gamma_{p^*}(t_m))\}_{j=1}^{M}$ and $\{\gamma_{p^*}(t_m)\}_{m=1}^{M}$, respectively,  via FFT-based differentiation algorithms~\cite{johnson2011notes}.  Note that since there are no quadrature points~\eqref{eq:Cheb_points} at the ends of the interval $[-1,1]$, the curve regularity requirements for the construction of $P_N$ are fulfilled at all the discretization points $\{\gamma_p(t_m)\}_{m=1}^M$ for $p=1,\ldots, P.$

As is well-known the accuracy of the numerically approximated derivatives deteriorates significantly as both the number~$M$ of Chebyshev points and the derivative order $N$ increases. In general, this phenomenon occurs because numerical differentiation is an ill-posed problem in the sense that small errors in the functions values, such as those stemming from round-off errors, give rise to large errors in the approximate derivatives. In order to circumvent this issue (which is also present in the context of the trapezoidal rule~\eqref{eq:trap_rule}), we divide the parameter space $[-1,1]$ into a suitable number of subintervals thus generating  smaller patches within which just a small number~$M$ of Chebyshev points  is needed to accurately perform both integration and differentiation. This strategy ensures that numerical differentiation of $\varphi\circ\gamma_{p^*}$ and $\gamma_{p^*}$ is carried out by differentiation of low-degree Chebyshev interpolation polynomials, which are significantly less affected by ill-conditioning~\cite{bruno2012numerical}.

\subsection{Determining the coefficients}\label{sec:det_coef}
This section presents a straightforward numerical procedure for computing the set of coefficients $\{c_j(z)\}_{j=0}^N$---that define the  density interpolant $P_N(\cdot,z)$ in~\eqref{eq:truc_taylor}---at points~$z\in\Gamma$ corresponding to the quadrature nodes introduced in Sections~\ref{sec:smooth} and~\ref{sec:psmooth}. 

As discussed, the discrete set $\{\phi(t_m)\}_{m=1}^{M}$, where $\{t_m\}_{m=1}^M$ could be either~\eqref{eq:uniform}  or~\eqref{eq:Cheb_points}, can be utilized to produce spectrally accurate approximations of the derivative values $\{\phi'(t_m)\}_{m=1}^{M}$ of a smooth function $\phi$ defined in the parameter space, by means of FFT-based differentiation. Let then $D_{\rm  FFT}:\C^{M}\to \C^M$ denote the linear transformation that produces  $(D_{\rm FFT}\pmb{\phi})_m \approx \phi'(t_m)$, $m=1,\ldots,M$, via FFT-based differentiation, where $\pmb{\phi} = [\phi(t_1),\ldots,\phi(t_M)]^T\in\C^M$.  Then the set of coefficients $\{c_j(\gamma(t_m))\}_{j=0}^N$ associated with each one of the quadrature nodes $t_m$, $m=1,\ldots M$, can  be   computed all at once by means of the following recursive procedure:\bigskip

\begin{algorithm}[H]
 \KwData{Sample values of the density function $\pmb{\phi}=[\phi(t_1),\ldots,\phi(t_M)]^T\in\C^M$;  $\pmb{\gamma}=[\gamma(t_1),\ldots,\gamma(t_M)]^{T}\in\C^M$; and density interpolation order $N\geq 0$\;}
 \KwResult{Matrix  $B\in\C^{M\times (N+1)}$ with entries $B(m,j+1)\approx c_{j}(\gamma(t_m))$, $m=1,\ldots M$, $j=0,\ldots,N$\;}
  \BlankLine
 $\pmb{\gamma}' = D_{\rm FFT}(\pmb{\gamma})$\; 
$B(:,1)=\pmb{\phi}$\; 
  \lIf{$N=0$}{\Return{B}}
 \For{$j$ from 1 to $N$}{
$ B(:,j+1)=D_{\rm FFT}\lf(B(:,j)\rg)\oslash\pmb{\gamma}'$\; \tcp{the symbol $\oslash$ denotes element-wise division}}
\Return{$B$}\;
 \caption{Evaluation of the coefficients in the definition of the density interpolant~\eqref{eq:truc_taylor}.}\label{alg:algo}
\end{algorithm}\bigskip

This simple procedure is based on the identity $P_N(z,z_m) = p_N(\gamma^{-1}(z),t_m)$, where $\gamma$ could be either a global or a local parametrization of the curve $\Gamma$, and $z_m=\gamma(t_m)$. Differentiating this identity $j$-times with respect to $z$, we get
$$
\lf.\frac{\p^j }{\p z^j}P_N(z,z_m)\rg|_{z=z_m} = \lf.D^j_\gamma p_N(t,t_m)\rg|_{t=t_m},\quad\mbox{where}\quad D_\gamma=\frac{1}{\gamma'(t)}\frac{\p}{\p t}.
$$
Therefore, using the identities
$$c_j(z_m)=\left.\frac{\p^j}{\p z^j}P_N(z,z_0)\right|_{z=z_m}\andtext \lf.D^j_\gamma p_N(t,t_m)\rg|_{t=t_m}=D^j_\gamma\phi(t_m),$$
that follow from~\eqref{eq:truc_taylor} and~\eqref{eq:inter_cond}, respectively, we  obtain  
$$
c_j(z_m) = D^j_\gamma\phi(t_m), \quad j=0,\ldots,N,\quad m=1,\ldots,M,
$$
which is the formula that Algorithm~\ref{alg:algo} implements.

Finally, we comment on the computational complexity of Algorithm~\ref{alg:algo}. Clearly,  the overall cost of constructing the $N$th-order density interpolants associated to each one of the discretization points on the contour, amounts to $O(NM\log (M))$ in the case a smooth curves, where $M$ is the total number of discretization points, and it amounts to $O(PNM\log(M))$ in the case of piecewise smooth curves, where~$P$ is the number of patches and~$M$ is number discretization points per patch.

\section{Examples} \label{sec:examples}
This final section presents a variety of numerical examples designed to validate and showcase the applicability of the proposed methodology.

\subsection{Validation} 
We start off by considering the smooth  ``jellyfish''  and the piecewise smooth ``snowflake'' contours displayed in Figures~\ref{fig:jellyfish} and~\ref{fig:snowflake}, respectively.   The jellyfish contour is given by the (real) analytic  parametrization 
\begin{equation}\label{eq:curve_param}
\gamma(t)=\{1+0.3\cos(4t+2\sin t)\}\e^{i(t-\frac{\pi}{2})},\quad t\in[0,2\pi),
\end{equation} while the snowflake curve corresponds to the Koch polygonal domain~\cite{falconer2004fractal} comprising~192 vertices. Here we validate the density interpolation technique developed in Section~\ref{sec:nearly_singular} and the high-order numerical methods presented in Section~\ref{sec:numerics} by using the Cauchy integral formula~\eqref{eq:CF}

In our first example, the Cauchy integral operator~$\mathcal Cf$ and its first two derivatives are evaluated at a set of points $\{w_k\}_{k=1}^{100}\in\Omega$ which are uniformly distributed along $\Gamma$ and placed at a distance~$10^{-4}$ or shorter from it.  Meromorphic  functions of the form~$f(z) = \sum_{\ell =1}^L (z-z_\ell)^{-1}$, with poles lying outside the domain~$\Omega$ bounded by~$\Gamma$ (the location of the poles $z_\ell$ of~$f$ used in this example are marked by the green dots in Figures~\ref{fig:jellyfish} and~\ref{fig:snowflake}), are used as input densities.  
The numerical errors in the evaluation of the operators~\eqref{eq:reg_formula} and~\eqref{eq:reg_formula_der} are then assessed by comparing the numerically produced values of $(\mathcal Cf)(w_k)$, $(\mathcal Cf)'(w_k)$, and $(\mathcal Cf)''(w_k)$, with the corresponding exact ones~$f(w_k)$, $f'(w_k)$ and $f''(w_k)$, for $k=1,\ldots, 100$. 
Table~\ref{tab:1} shows the results corresponding to the maximum relative errors achieved without regularization and using the density interpolation method of orders $N=0,\ldots, 5$. A total of $M=800$ quadrature nodes were used in the case of the jellyfish contour which is discretized applying the numerical integration/differentiation approach presented in Section~\ref{sec:smooth}. The snowflake curve, in turn,  was discretized using the approach presented in Section~\ref{sec:psmooth} using a total of $P=576$ patches with~$M=8$ quadrature nodes per patch.   

Figures~\ref{fig:jellyfish} and~\ref{fig:snowflake}, on the other hand, display the logarithm in base ten of the absolute error~$|f(z)-\tilde f(z)|$ for $z\in\Omega$, in the evaluation of the Cauchy operator for various density interpolation orders. The ``$5h$" rule of thumb~\cite{Barnett:2014tq} was used in this example, where the regularization was applied at  points $z\in\Omega$ at a distance smaller than $5h$ from the contour where $h$ is an approximation of   distance between the contour discretization points that are the closest to~$z$. A total of $M=400$ quadrature nodes were used in the case of the jellyfish contour, and $P=576$ patches and $M=6$ points per patch in the case of the snowflake curve.


\begin{table}
\caption{Relative errors $E_n = \max_{1\leq k\leq 100}|f^{(n)}(w_k)-\tilde f^{(n)}(w_k)|/|f^{(n)}(w_k)|$, with~$f^{(n)}(w_k)$ and $\tilde f^{(n)}(w_k)$ denoting the exact and approximate values, respectively, in the evaluation of the Cauchy integral operator ($n=0$) and its first ($n=1$)  and second ($n=2$) order derivatives, at a fixed set of points $\{w_k\}_{k=1}^{100}\in\Omega$ placed at a distance~$10^{-4}$ or shorter from the contour~$\Gamma$. The table reports the results obtained without using regularization and using the density interpolation method of order $N=0,\ldots,4$, for the (smooth) jellyfish and the (polygonal) snowflake contours displayed in Figures~\ref{fig:jellyfish} and~\ref{fig:snowflake}, respectively. }
\begin{center}
\begin{tabular}{ c|c|c|c|c|c|c } 
Reg. order & \multicolumn{3}{|c}{Errors (jellyfish)} & \multicolumn{3}{|c}{Errors (snowflake)} \\ 
\toprule
$N$& $E_0$&$ E_1$&$ E_2$& $E_0$&$ E_1$&$ E_2$\\
 \toprule
without & $2.48\cdot 10^{+01}$ & $3.98\cdot 10^{+04}$&$1.11\cdot 10^{+08}$& $3.24\cdot 10^{+00}$&$3.05\cdot 10^{+03}$&$2.75\cdot 10^{+06}$\\ 
 0th & $1.60\cdot 10^{-02}$ & $5.04\cdot 10^{-01}$ &$1.35\cdot 10^{+01}$&$4.02\cdot 10^{-03}$&$4.60\cdot 10^{-01}$&$9.91\cdot 10^{+02}$\\ 
 1st  & $5.72\cdot 10^{-06}$ & $8.79\cdot 10^{-03}$ &$5.02\cdot 10^{-01}$&$4.47\cdot 10^{-06}$&$3.18\cdot 10^{-03}$&$4.38\cdot 10^{-01}$\\ 
 2nd  & $2.43\cdot 10^{-09}$ & $7.57\cdot 10^{-06}$&$1.03\cdot 10^{-02}$&$5.20\cdot 10^{-09}$&$7.98\cdot 10^{-06}$&$2.81\cdot 10^{-03}$ \\ 
 3rd  & $1.04\cdot 10^{-12}$ & $4.55\cdot 10^{-09}$ &$1.27\cdot 10^{-05}$&$7.67\cdot 10^{-12}$&$1.43\cdot 10^{-08}$&$1.14\cdot 10^{-05}$\\ 
 4th  & $3.54\cdot 10^{-13}$ & $1.46\cdot 10^{-10}$ &$2.10\cdot 10^{-07}$&$3.52\cdot 10^{-12}$&$2.68\cdot 10^{-10}$&$4.27\cdot 10^{-08}$\\ 
\bottomrule
\end{tabular}
\end{center}\label{tab:1}
\end{table}

\begin{figure}[ht]
\centering	
\subfloat[Without regularization. $E=1.27\cdot 10^1$.]{\includegraphics[height=0.3\textwidth]{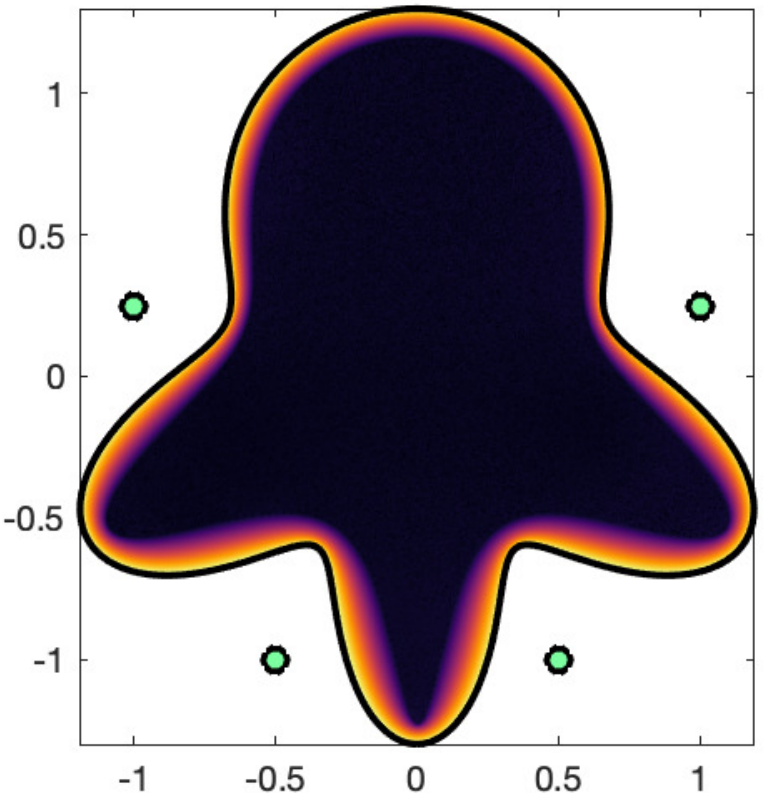}}\quad
\subfloat[$0$th order.  $E=1.18\cdot 10^{-1}$.]{\includegraphics[height=0.3\textwidth]{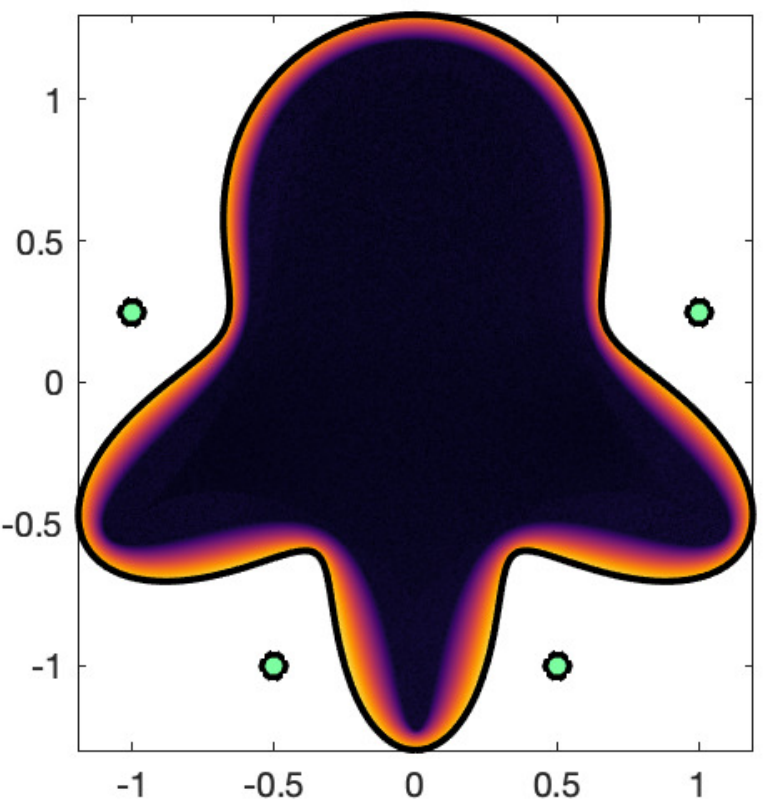}}\quad
\subfloat[ $1$st order. $E=6.40\cdot 10^{-3}$.]{\includegraphics[height=0.3\textwidth]{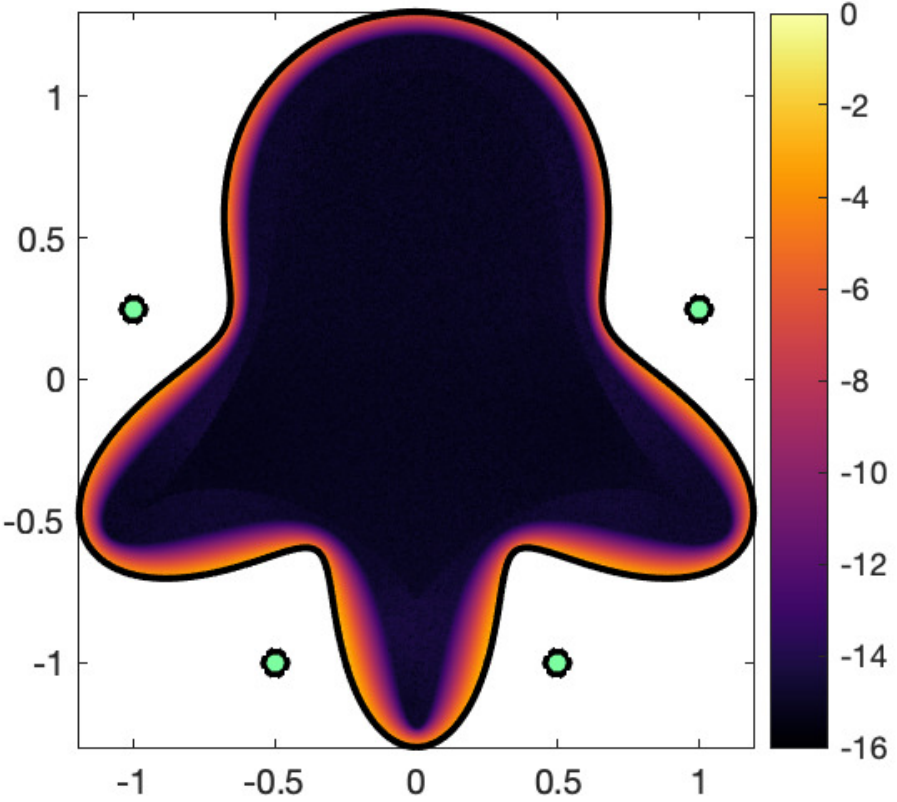}}\\
 \subfloat[$2$nd order. $E=3.85\cdot 10^{-4}$.]{\includegraphics[height=0.3\textwidth]{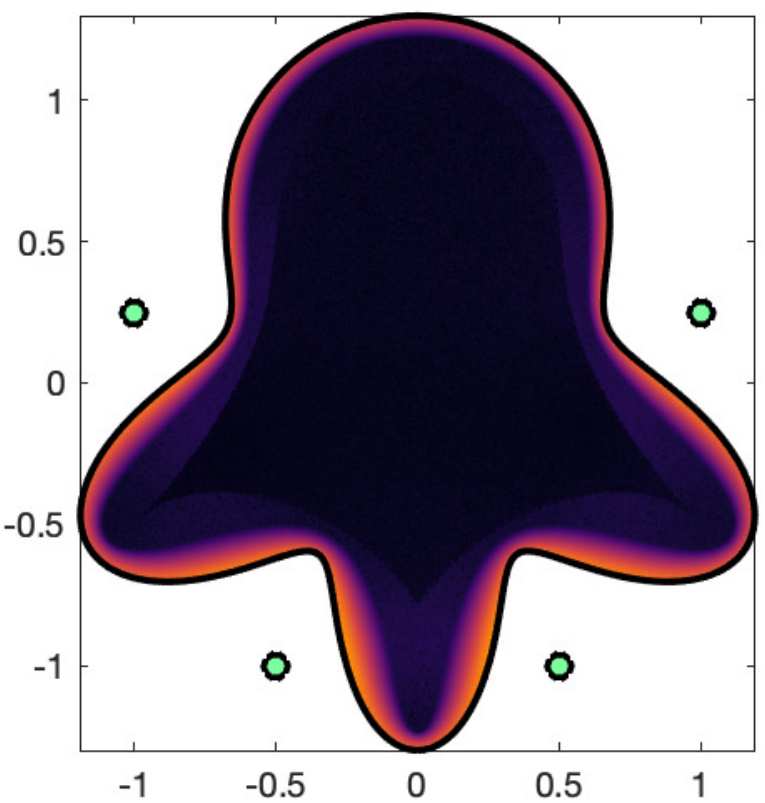}}\quad
 \subfloat[$3$rd order. $E=2.27\cdot 10^{-5}$.]{\includegraphics[height=0.3\textwidth]{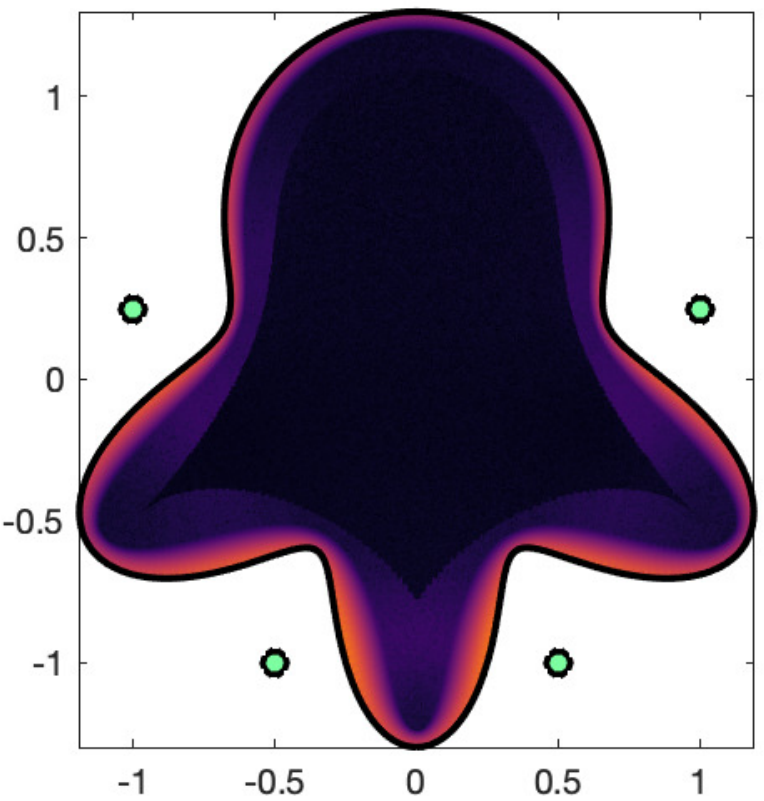}}\quad
 \subfloat[a][$4$th order. $E=1.38\cdot 10^{-6}$.]{\includegraphics[height=0.3\textwidth]{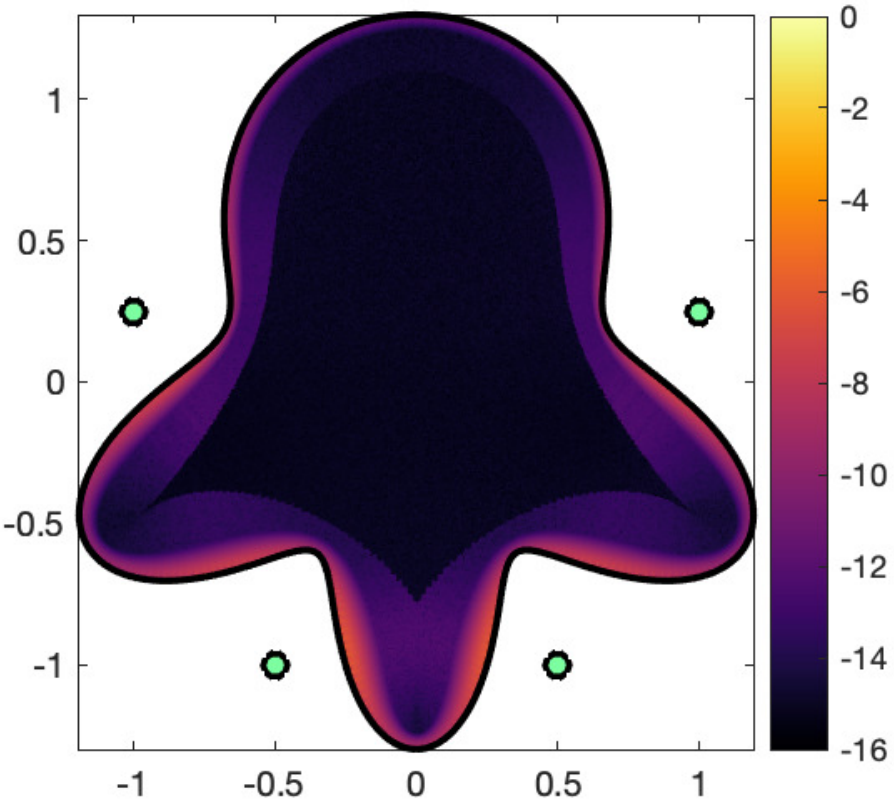}}
 \caption{Logarithm in base ten of the absolute error in the
   evaluation of the Cauchy operator inside a jellyfish smooth contour for density interpolation orders $N=0,1,2,3$ and~$4$. The input function used corresponds to the contour restriction of a meromorphic function with poles at the locations marked by the green dots. The maximum absolute error, $E$, is provided  in the captions.}\label{fig:jellyfish}
\end{figure}

\begin{figure}[ht]
\centering	
\subfloat[Without regularization. $E=5.45\cdot 10^1$.]{\includegraphics[height=0.3\textwidth]{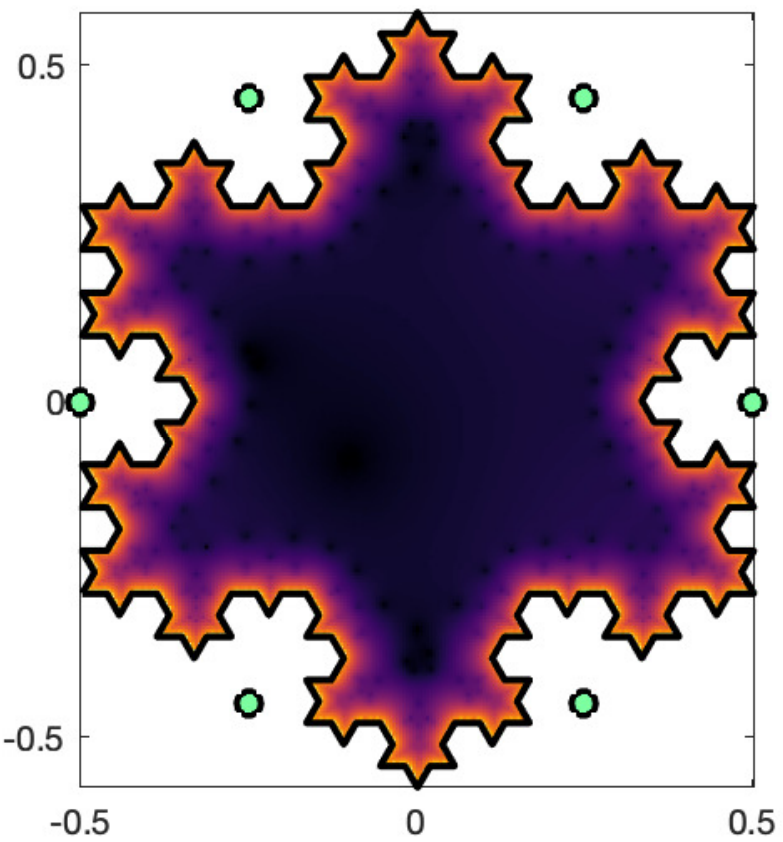}\label{fig_D_none}}\quad
\subfloat[$N=0$.  $E=6.00\cdot 10^{-2}$.]{\includegraphics[height=0.3\textwidth]{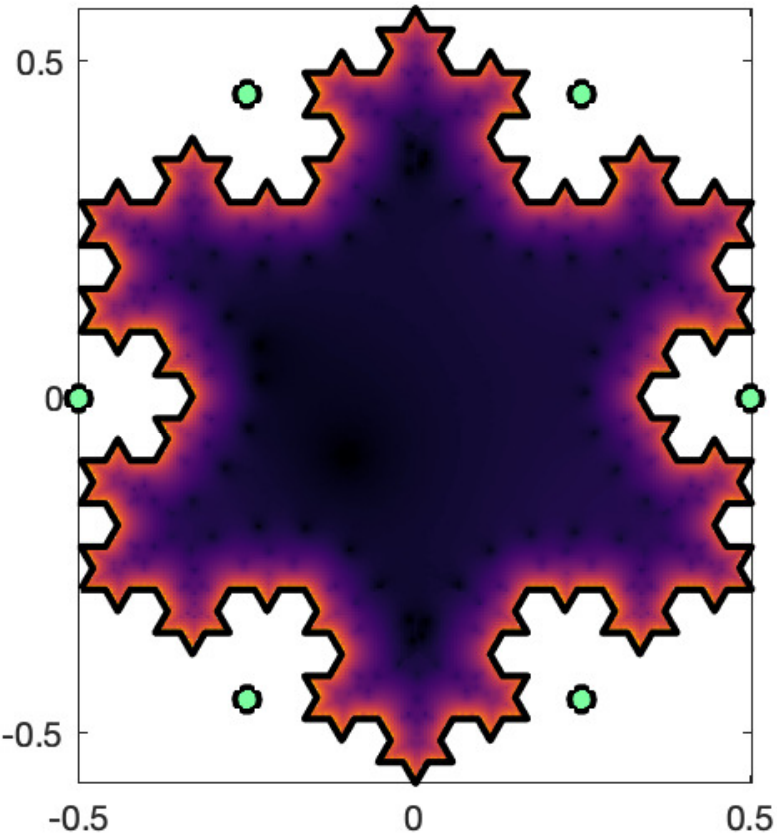}\label{fig_D_0}}\quad
\subfloat[$N=1$. $E=5.18\cdot 10^{-4}$.]{\includegraphics[height=0.3\textwidth]{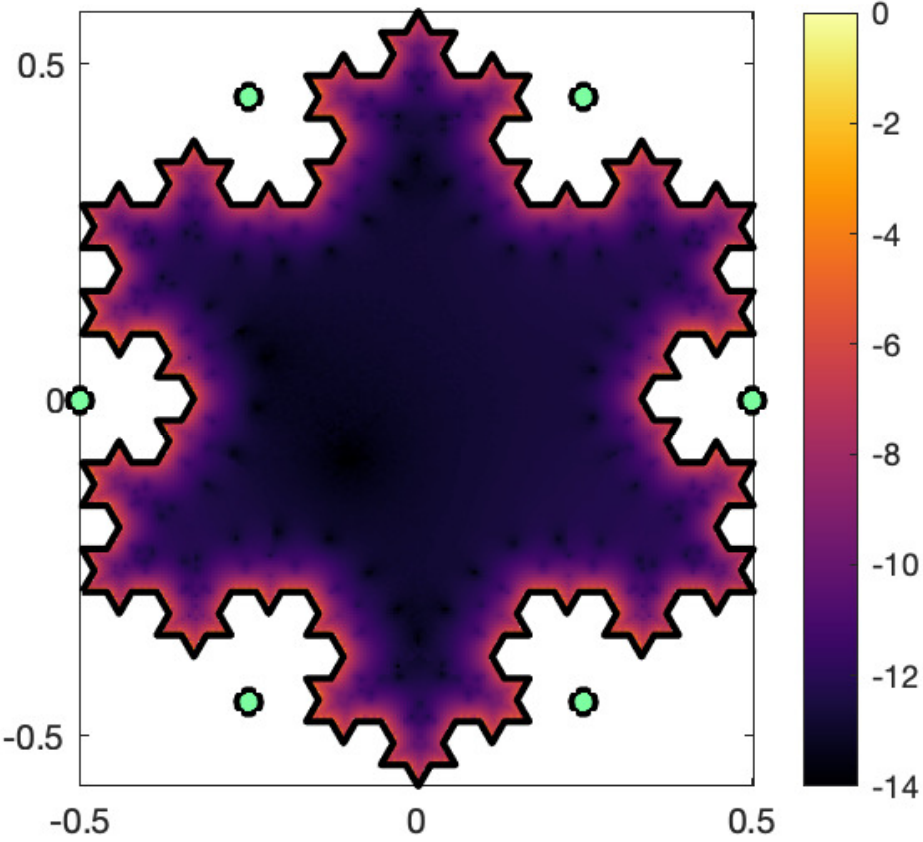}\label{fig_D_4}}\\
 \subfloat[$N=2$. $E=1.03\cdot 10^{-5}$.]{\includegraphics[height=0.3\textwidth]{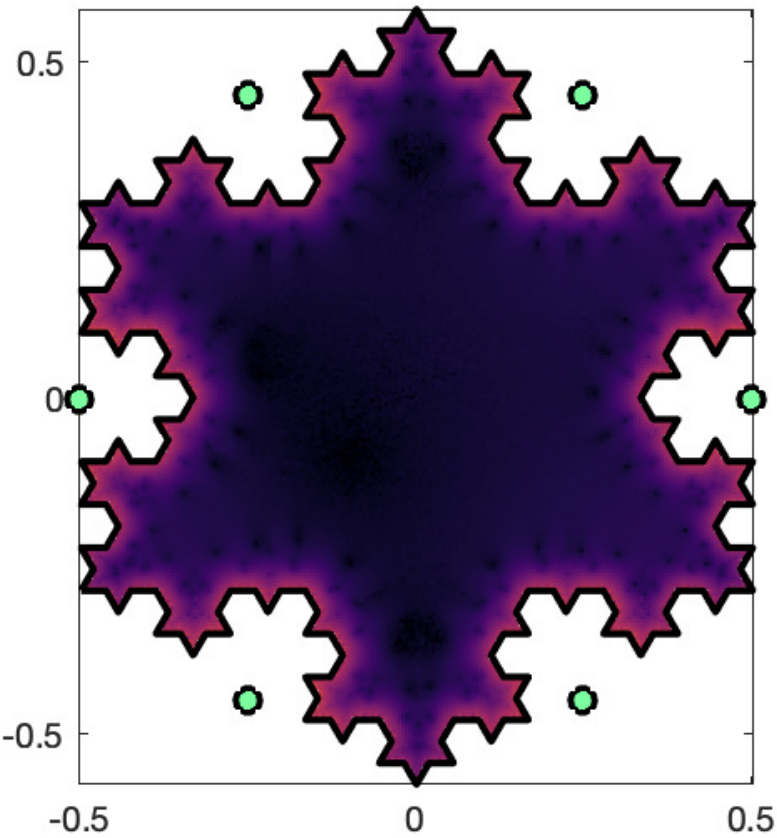}\label{fig_grad_D_none}}\quad
 \subfloat[$N=3$. $E=2.43\cdot 10^{-7}$.]{\includegraphics[height=0.3\textwidth]{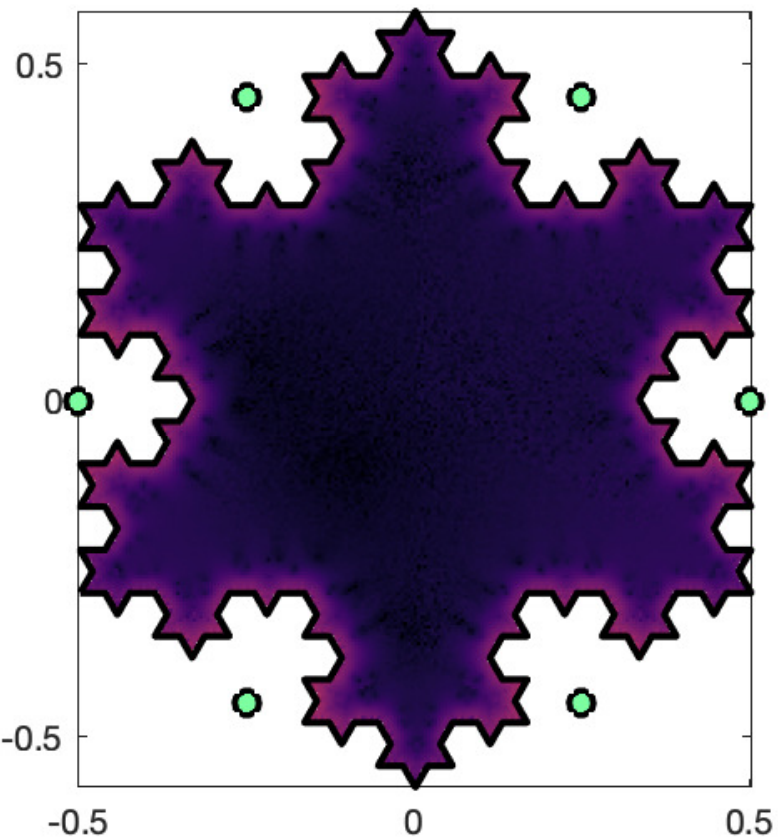}\label{fig_grad_D_0}}\quad
 \subfloat[a][$N=4$. $\max E=1.37\cdot 10^{-8}$.]{\includegraphics[height=0.3\textwidth]{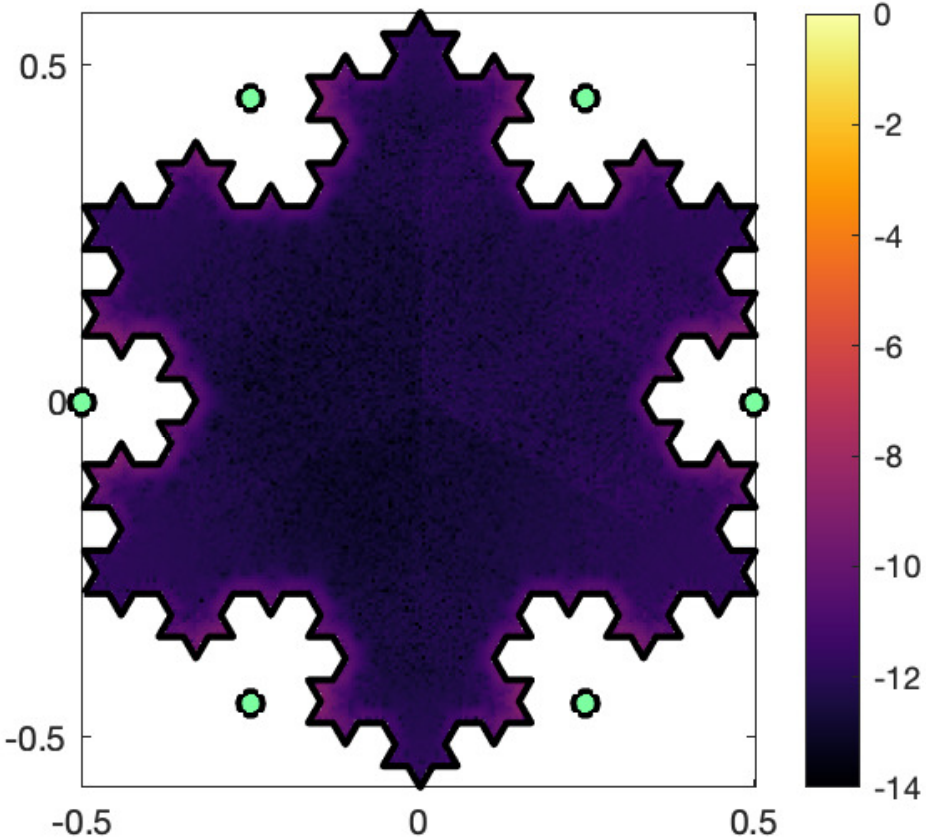}\label{fig_grad_D_4}}
 \caption{Logarithm in base ten of the absolute error in the
   evaluation of the Cauchy operator inside the Koch polygonal curve for density interpolation orders $N=0,1,2,3$ and~$4$. The input function used corresponds to the contour restriction of a meromorphic function with poles at the locations marked by the green dots. The maximum absolute error, $E$, is provided in the captions. }\label{fig:snowflake}
\end{figure}  

\begin{figure}[ht]
\centering	
\subfloat[Integral equation~\eqref{eq:IE_SL}.]{\includegraphics[height=0.4\textwidth]{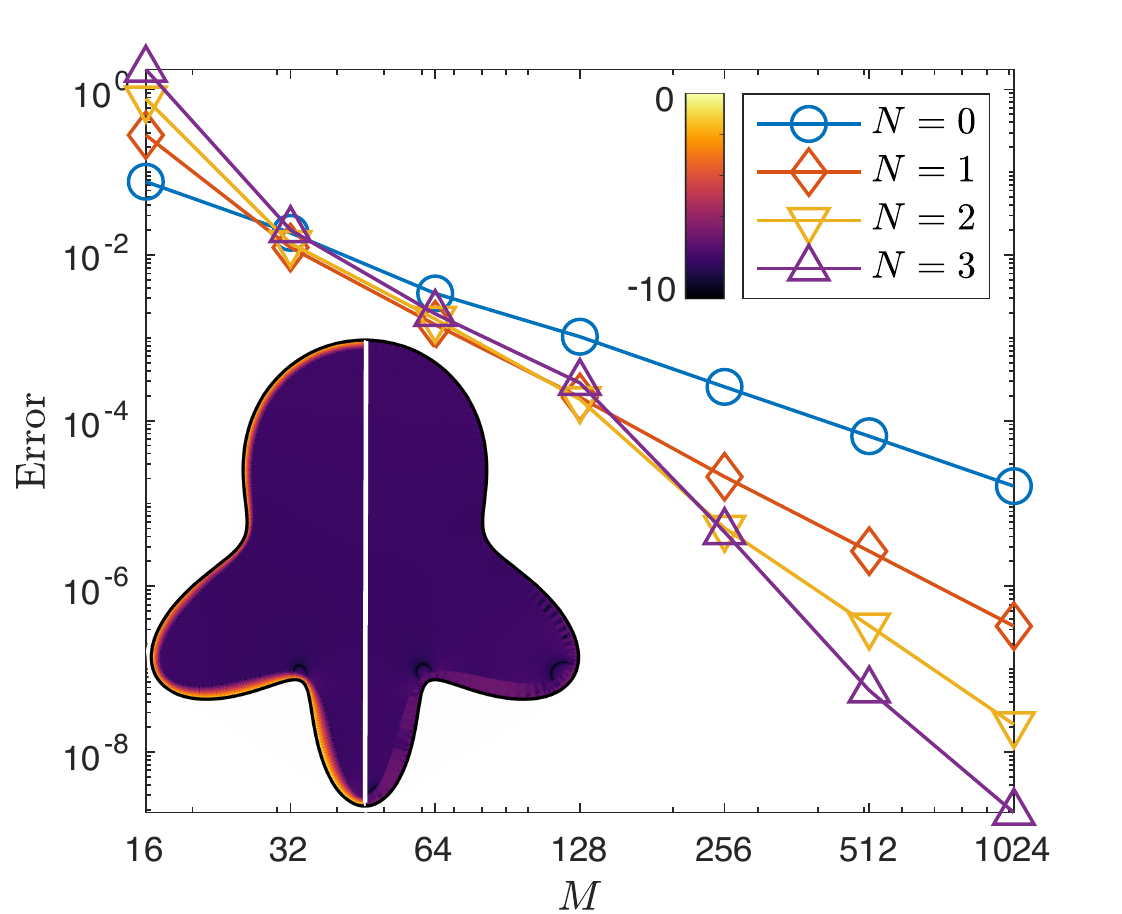}\label{fig_Sop}}
\subfloat[Integral equation~\eqref{eq:IE_HS}.]{\includegraphics[height=0.4\textwidth]{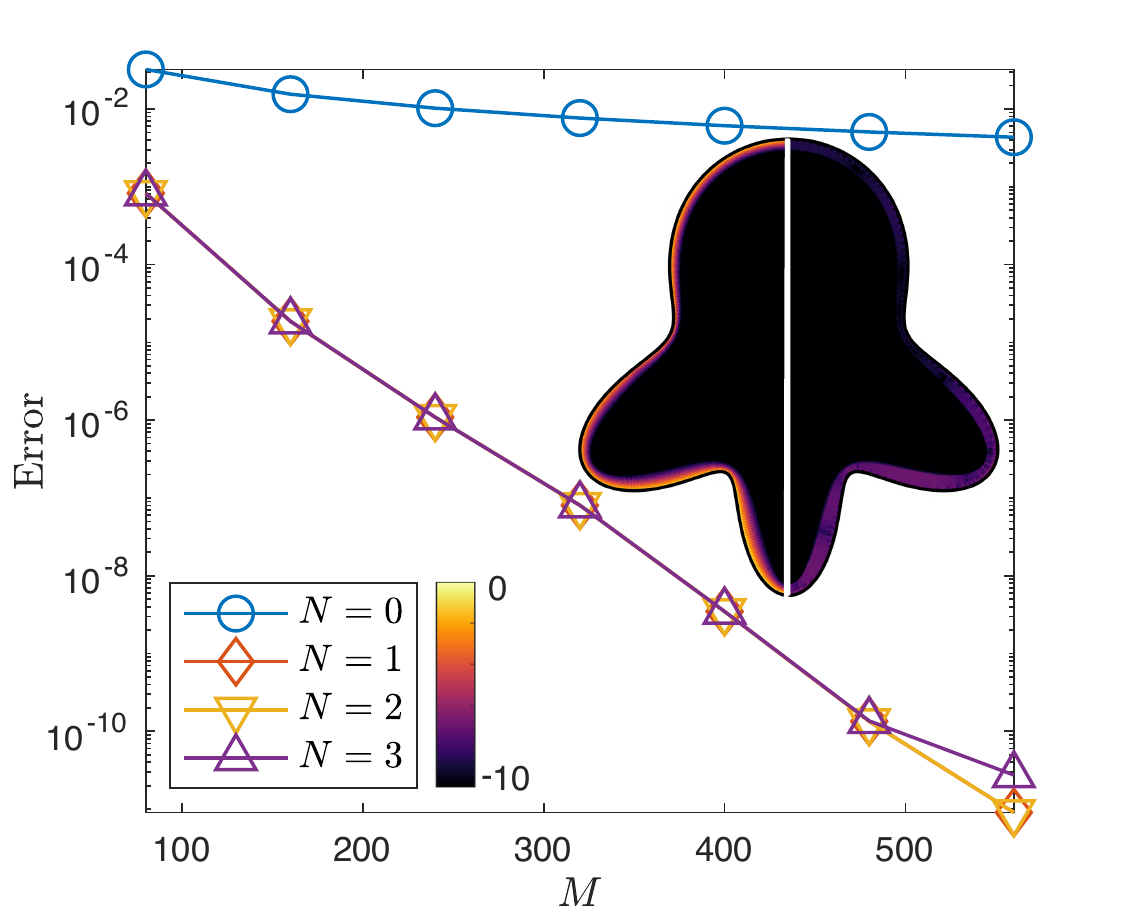}\label{fig_Top}}
 \caption{Relative errors in the solution of the integral equations~\eqref{eq:IE_SL}  and~\eqref{eq:IE_HS}, corresponding to the plots  (a) in log-log scale and  (b) in semi-log scale, respectively, computed using the trapezoidal rule based Nystr\"om method for various numbers~$M$ of discretization points. The inset figures display the logarithm in base ten of of the absolute error in the numerically approximated solution $u$ inside the domain, without using regularization (left half of the inset figure) and using the density interpolation of order $N=3$ (right half of the inset figure). }\label{fig:IE_convergence}
\end{figure}

\subsection{Laplace equation layer potentials and boundary integral operators}
Following the discussion in Section~\ref{sec:DL_formulation}, we here apply the density interpolation method to solve the Laplace equation by means of boundary integral equation methods. 

For the sake of completeness, we consider the uniquely solvable interior Robin  problem:
\begin{equation}\label{eq:interior_Laplace}
\Delta u =0\quad\mbox{in}\quad\Omega,\qquad \frac{\p u}{\p \nu} + u = f\quad\mbox{on}\quad\Gamma,
\end{equation}
which is formulated as direct boundary integral equations involving the two layer potentials and all four integral operators of Calder\'on calculus. The domain's boundary $\Gamma=\p\Omega$ is assumed of class~$C^2$ and $f\in C^{1,\alpha}(\Gamma)$, $0<\alpha<1$. Integral equations for  the unknown traces $\p u/\p \nu$ and $u$ on~$\Gamma$ are derived from the Green's representation formula
\begin{equation}\label{eq:GF}
u(\ner) = \lf(\mathcal S\frac{\p u}{\p \nu}\rg)(\ner) - (\mathcal D u)(\ner),\quad \ner\in\Omega.
\end{equation}
Indeed, evaluating this formula 
on $\Gamma$---by making use of the interior jump conditions for the potentials~\cite{kress2012linear} and employing the boundary condition---we obtain the following second-kind integral equation for the unknown normal derivative of the solution:
\begin{equation}\label{eq:IE_SL}
\lf(\frac{I}{2}+K+S\rg)\frac{\p u}{\p \nu} = \lf(\frac{I}{2}+K\rg)f\quad\mbox{on}\quad\Gamma,
\end{equation}
where $I$, $S$ and $K$ are respectively the identity operator, the single-layer operator~\eqref{eq:SL_op}, and the double-layer operator~\eqref{eq:DL_op}.
In a similar manner, taking  the normal derivative of the Green's formula on~$\Gamma$---by making use of the exterior jump conditions of the layer potentials gradients~\cite{kress2012linear}---and enforcing the boundary condition, we obtain
\begin{equation}\label{eq:IE_HS}
\lf(-\frac{I}{2}+K^\top+T\rg)u = \lf(-\frac{I}{2}+K^\top\rg)f\quad\mbox{on}\quad\Gamma,
\end{equation}
for the unknown solution $u$ on $\Gamma$, where $K^\top$ and $T$ are respectively the adjoint double-layer operator~\eqref{eq:ADL_op} and the hypersingular operator~\eqref{eq:hypersingular}.

The integral equations~\eqref{eq:IE_SL} and~\eqref{eq:IE_HS} are  discretized following a Nystr\"om method~\cite{kress2012linear} based on direct use of the trapezoidal-rule discretization described in  Section~\ref{sec:smooth}. The required density interpolants are constructed following the procedures outlined in Sections~\ref{sec:smooth} and~\ref{sec:det_coef}. Since the single-layer ($S$) and hypersingular ($T$) operators are recast, in~\eqref{eq:SL_op_reg} and~\eqref{eq:hyper_complex}, respectively, in terms of smooth integrands, the trapezoidal rule yields the expected order of convergence  according to the achieved smoothness of the integrands, which, in turn, depends on the density interpolation order $N$. Note that the remaining double-layer ($K$) operator and its adjoint ($K^\top$) do not need regularization in this case. The resulting linear systems for the approximate values $v_j$ and $u_j$ of the traces $\p u/\p\nu$ and $u$ at the quadrature nodes $\nex_j\in\Gamma$, $j=1,\ldots,$ are iteratively solved by means of GMRES~\cite{saad1986gmres}, which only requires forward map evaluations of the integral operators.

In order to examine the accuracy of the Nystr\"om method, we let $\Gamma$ be the smooth jellyfish curve given by the parametrization in~\eqref{eq:curve_param}, and $f$ be so that $u(x,y)=\e^x\sin y$ is the exact solution of~\eqref{eq:interior_Laplace}. Figure~\ref{fig_Sop} displays the relative errors $\max_{1\leq j\leq M}|v_j-\frac{\p u}{\p \nu}(\nex_j)|/\max_{1\leq j\leq M}|\frac{\p u}{\p \nu}(\nex_j)|$ in the numerical solution of the second-kind integral equation~\eqref{eq:IE_SL}. Numerical errors of order $O(M^{-N-2})$ as $M\to \infty$, for interpolation orders $N=0,\ldots, 3,$, are observed in these examples. The convergence appears to be slightly delayed due to the significant curvature of the jellyfish contour used, which requires a relatively large number of discretization points $M$ to be properly resolved. The observed convergence orders are explained by the fact that approximate second-kind integral equation solutions obtained by means of Nystr\"om methods, inherit the accuracy of the associated quadrature rule~\cite{atkinson1997numerical,hackbusch1995}. Therefore, since the dominant quadrature error in the approximation of $(I/2+K+S)\varphi$ at the nodes, stems from the evaluation of the regularized single-layer operator $S$ in~\eqref{eq:SL_op_reg} (since $K$ features an analytic kernel)
and, as we show bellow,  the direct trapezoidal rule approximation~\eqref{eq:trap_rule_total} of $S\varphi$ yields $O(M^{-N-2})$ errors as $M\to\infty$ for analytic contours and densities, we achieve the same asymptotic errors in the integral equation solution. 

The following results establish the abovementioned  asymptotic error bound for the trapezoidal rule approximation of the regularized single-layer operator $S$:

\begin{lemma} \label{lem:trap_rule}Let $g$ be an analytic function on $[0,2\pi]$  such that $g^{(n)}(0)=g^{(n)}(2\pi)=0$ for $n=0,\ldots N$. Then
\begin{eqnarray}
\left|\int_{0}^{2\pi}g(t)\de t-\frac{2\pi}{M}\sum_{m=1}^{M}g(t_m)\right|\leq O\left(M^{-N-2}\right),\label{eq:trap_smooth}\\
\left|\int_{0}^{2\pi}g(t)\log\left(4\sin^2\frac{t}{2}\right)\de t-\frac{2\pi}{M}\sum_{m=2}^{M}g(t_m)\log\left(4\sin^2\frac{t_m}{2}\right)\right|\leq  O\left(M^{-N-2}\right),\label{eq:trap_log}
\end{eqnarray}
as $M\to\infty$, where $t_m=\frac{2\pi}{M}(m-1)$ for $m=1,\ldots,M.$
\end{lemma}
\begin{proof}
The error bound~\eqref{eq:trap_smooth} follows directly from~\cite[Corollary 3.3]{javed2014trapezoidal}, while~\eqref{eq:trap_log}, in turn, follows by  slightly modifying  the proof of the Euler-Maclaurin expansion for functions with logarithmic singularities presented in~\cite[Theorem 1]{celorrio1999euler}.
\end{proof}
\begin{theorem}  Let $\gamma$ and $\varphi\circ\gamma$ be analytic and $2\pi$-periodic functions on $[0,2\pi]$, and let $\psi$, $Q_N$ and $\log(\cdot-z_0)$, with $z_0=\gamma(0)\in\Gamma$, be the functions defined in Theorem~\ref{th:single_layer}. Then, the error in the  direct trapezoidal rule approximation~\eqref{eq:trap_rule_total} of the regularized single-layer operator~\eqref{eq:SL_op_reg} at $\nex=(\real z_0,\imag z_0)$, satisfies:
$$
\left|(S\varphi)(\nex)+\real\lf\{\frac{1}{M}\sum_{m=2}^{M} \log(\gamma(t_m)-\gamma(0))\lf\{\psi(\gamma(t_m))-Q_N(\gamma(t_m),\gamma(0))\rg\}\gamma'(t_m)\rg\}\right|\leq O(M^{-N-2}),
$$
as $M\to\infty$.
\end{theorem}
\begin{proof} Let
\begin{equation}\label{eq:thm_ex0}
g_0(t) = \lf\{\psi(\gamma(t))-Q_N(\gamma(t),\gamma(0))\rg\}\gamma'(t),
\end{equation}
which is $2\pi$-periodic and analytic on $[0,2\pi]$, and note that  $g^{(n)}_0(0)=g^{(n)}_0(2\pi)=0$ for $n=0,\ldots,N$, because $Q_N(\cdot, z_0)$ is the density interpolant of $\psi$ at $z_0=\gamma(0)$.
On the other hand, we have the identity
\begin{equation}\label{eq:thm_ex}
\log(\gamma(t)-\gamma(0)) =\frac{1}{2}\log\lf(4\sin^2 \frac{t}{2}\rg)+g_1(t),
\end{equation}
where 
$$
 g_1(t)=\frac{1}{2}\log\lf(\frac{|\gamma(t)-\gamma(0)|^2}{4\sin^2 \frac{t}{2}} \rg)+i\operatorname{arg}(\gamma(t)-\gamma(0)),
$$
 is an analytic (but not periodic) function on $[0,2\pi]$.
Therefore, it follows from~\eqref{eq:thm_ex0} and~\eqref{eq:thm_ex} that the  single-layer operator can be expressed as
$$
(S\varphi)(\nex) = -\frac{1}{2\pi}\real\left\{\frac{1}{2}\int_0^{2\pi}g_0(t)\log\lf(4\sin^2 \frac{t}{2}\rg)\de t+\int_0^{2\pi}g_0(t)g_1(t)\de t\right\}.
$$
Since $g=g_0g_1$ satisfies
$$
g^{(n)}(0) =\sum_{k=0}^{n}{n \choose k}  g_1^{(n-k)}(0)  g_0^{(k)}(0)=\sum_{k=0}^{n}{n \choose k}  g_1^{(n-k)}(2\pi) g_0^{(k)}(2\pi) = g^{(n)}(2\pi) =0
$$
for $n=0,\ldots,N$, we obtain, from Lemma~\ref{lem:trap_rule}, the asymptotic error bound.
\end{proof}

Continuing with the numerical examples, Figure~\ref{fig_Top} displays the relative errors $\max_{1\leq j\leq M}|u_j- u(\nex_j)|/\max_{1\leq j\leq M}|u(\nex_j)|$ in the numerical solution of the integral equation~\eqref{eq:IE_HS}. Exponential convergence as $M\to \infty$ is achieved in this example for the interpolation orders $N=1,2$ and~3. This phenomenon is explained by the fact that, as the adjoint double-layer operator $K^\top$, the regularized hypersingular operator~\eqref{eq:hyper_complex} is given in terms of  $2\pi$-periodic analytic integrands for $N=1,2$ and~3, and, as such, the trapezoidal rule yields exponential convergence in the overall evaluation of $(-I/2+K^\top+T)\varphi$ at the quadrature nodes as $M$ increases. This eventually translates into the same exponential convergence of the integral equation solution.

The inset figures in~Figures~\ref{fig_Sop} and~\ref{fig_Top} show the logarithm in base ten of the absolute error in the numerical solution of~\eqref{eq:interior_Laplace} obtained from the Green's formula~\eqref{eq:GF} with (right half) and without (left half) regularization of the single- and double-layer potentials at points near the boundary. A total of $M=400$ discretization points and the interpolation order $N=3$ were used to produced these figures. The color difference between them (note that they are displayed using the same color scale) indicates superior accuracy achieved using the integral equation~\eqref{eq:IE_HS}.

 \begin{figure}[h!]
\centering	
\subfloat[Interior conformal mappings: Moai (top figure) and Nazca bird (bottom figure).]{\includegraphics[width=0.71\textwidth]{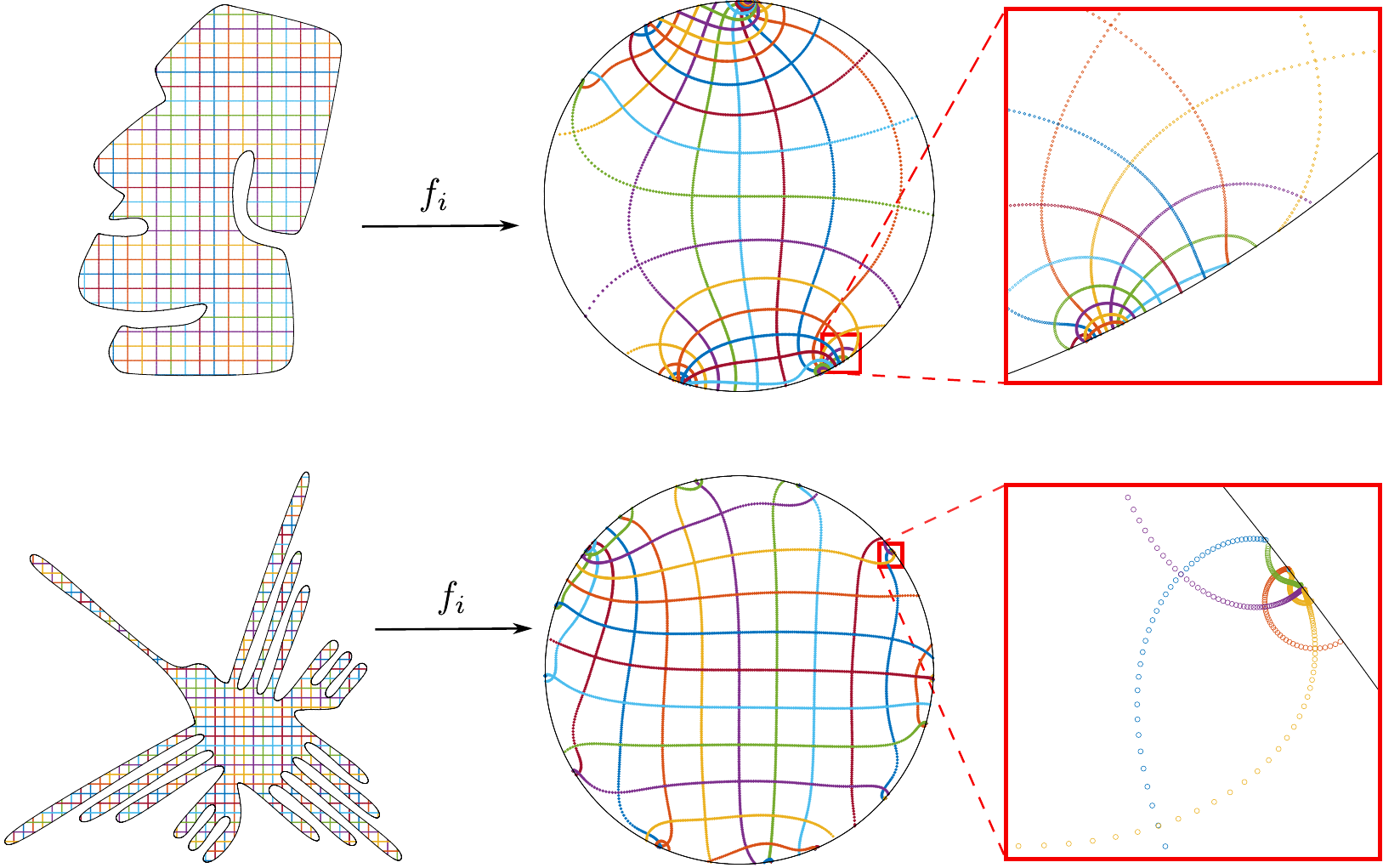}\label{fig:Moai}}\\
\subfloat[Exterior conformal mappings: Moai (top figure) and Nazca bird (bottom figure).]{\includegraphics[width=0.71\textwidth]{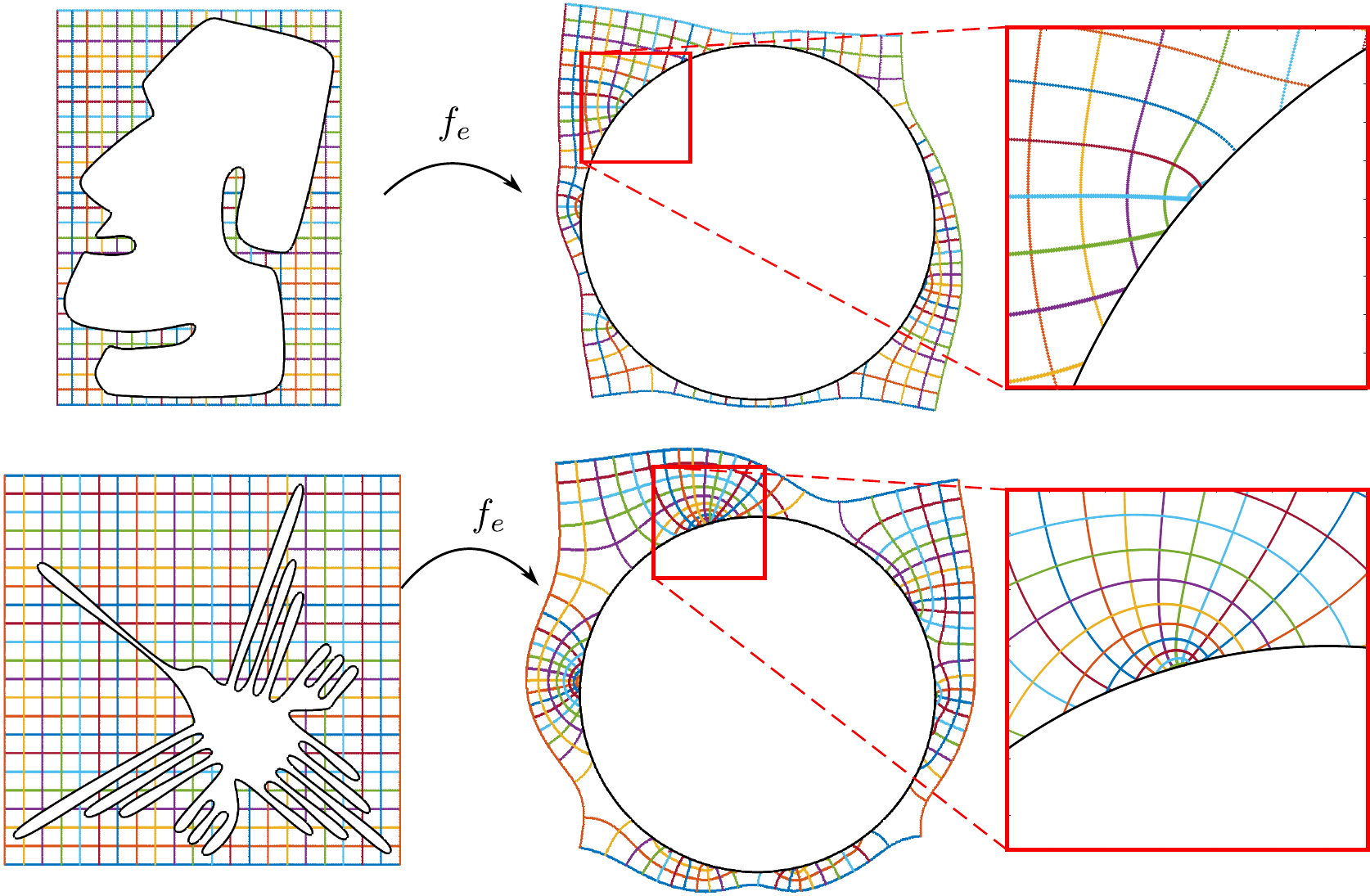}\label{fig:Nazca}}
 \caption{Examples of conformal mappings from $C^2$ domains to the interior~(a) and exterior~(b) of the unit disk. The interior and exterior mappings, given by $f_i$ in~\eqref{eq:conf_map} and $f_e$ in~\eqref{eq:ext_conf_map}, respectively, are regularized by the proposed density interpolation technique (at points $z$ lying near the contour) and numerically evaluated by means of the discretization approach presented in Section~\ref{sec:psmooth}. The zoomed-in figures 
demonstrate that the accuracy achieved by the density interpolation method in the evaluation of the mappings $f_i$ and $f_e$ allows them to retain their angle preserving property near the contour even at regions of large distortion.}\label{fig:cmap}
\end{figure}

In our final example we apply the density interpolation technique to the numerical evaluation of conformal mappings from the interior (resp. exterior) of $C^2$ curves into the interior (resp. exterior) of the unit disk. We consider here the \emph{Moai} and \emph{Nazca bird} curves depicted in Figures~\ref{fig:cmap} whose~$C^2$ parametrizations produced by cubic spline interpolation. The associated interior (resp. exterior) map integral equation~\eqref{eq:int_IE} (resp.~\eqref{eq:ext_IE}) is solved by applying a direct Nystr\"om method based on the Chebyshev discretization approach presented in Section~\ref{sec:psmooth}. The density interpolation method of order $N=3$ is then employed to evaluate the interior ($f_i$) and exterior ($f_e$) conformal mappings near the boundary; the zoomed figures provide evidence of the achieved accuracy in the numerical evaluation of  $f_i$ in~\eqref{eq:conf_map} and $f_e$ in~\eqref{eq:ext_conf_map} which are expressed terms of the Cauchy integral operator.

It is important to point out that in the cases of the Araucaria of Figure~\ref{fig:AraucariaExample}, and the snow	flake of Figure~\ref{fig:snowflake} and Table~\ref{tab:1}, which are the only contours in the paper that feature corners, the integrands considered are in fact smooth up to and including the endpoints of each curve segment. As such, the advocated Chebyshev discretization yields spectral accuracy in the evaluation of the resulting integral operator as the number of quadrature nodes increases. This is, however, often not the case for boundary integral equation solutions, which typically develop singularities depending on the angles at corners~\cite{kozlov1997elliptic}. As is well known, these singularities do have a significant impact on the overall numerical accuracy, that cannot be completely avoided by the quadratic refinement effected by the Chebyshev grids~\cite{Kress:1990vm,bruno2009high,anand2012well}. A treatment of such singularities in the context of density interpolation methods, which is compatible with the method proposed here, is given in~\cite[Sec.~5.3]{perez2018plane}. We mention, finally, that since both the Moai and the Nazca bird contours in Figure~\ref{fig:cmap}, are given in terms of parametrizations that are (globally) of class $C^2$, they do not feature corners. Therefore, although the associated conformal mappings were produced from Laplace integral equation solutions, the global regularity of the curves does not limits the accuracy of the Chebyshev grid discretization significantly.

\section{Conclusions}
We have introduced a new high-order density interpolation technique for the kernel regularization of Cauchy-like integral operators, such as  Laplace,  biharmonic, and Stokes layer potentials and associated  boundary integral operators in two dimensions, as well as contour integral representations of conformal mappings. The proposed methodology relies on the well-known Cauchy integral and Sokhotski-Plemelj formulae of complex analysis, together with local Taylor interpolations of input density functions along the contour. High-order numerical methods for the practical implementation of the proposed technique were presented for both smooth and piecewise smooth contours. An FFT-based algorithm applicable to both cases was introduced for the accurate and efficient  construction of the density interpolant. Application of this methodology to Stokes flows in the presence of moving interfaces is currently under investigation.

\bibliographystyle{abbrv}
\bibliography{References}

\begin{thebibliography}{10}

\bibitem{aldrovandi2001special}
R.~Aldrovandi.
\newblock {\em {Special Matrices of Mathematical Physics: Stochastic,
  Circulant, and Bell Matrices}}.
\newblock World Scientific, 2001.

\bibitem{anand2012well}
A.~Anand, J.~Ovall, and C.~Turc.
\newblock Well-conditioned boundary integral equations for two-dimensional
  sound-hard scattering problems in domains with corners.
\newblock {\em Journal of Integral Equations and Applications}, 24(3):321--358,
  2012.

\bibitem{atkinson1997numerical}
K.~E. Atkinson.
\newblock {\em The Numerical Solution of Integral Equations of the Second
  Kind}, volume~4.
\newblock Cambridge university press, 1997.

\bibitem{Barnett:2014tq}
A.~Barnett.
\newblock {Evaluation of layer potentials close to the boundary for Laplace and
  Helmholtz problems on analytic planar domains}.
\newblock {\em SIAM Journal on Scientific Computing}, 36(2):A427--A451, 2014.

\bibitem{Barnett:2015kg}
A.~Barnett, B.~Wu, and S.~Veerapaneni.
\newblock Spectrally accurate quadratures for evaluation of layer potentials
  close to the boundary for the {2D} {Stokes} and {Laplace} equations.
\newblock {\em SIAM Journal on Scientific Computing}, 37(4):B519--B542, Jan.
  2015.

\bibitem{bruno2012numerical}
O.~Bruno and D.~Hoch.
\newblock Numerical differentiation of approximated functions with limited
  order-of-accuracy deterioration.
\newblock {\em SIAM Journal on Numerical Analysis}, 50(3):1581--1603, 2012.

\bibitem{bruno2009high}
O.~P. Bruno, J.~S. Ovall, and C.~Turc.
\newblock A high-order integral algorithm for highly singular pde solutions in
  lipschitz domains.
\newblock {\em Computing}, 84(3-4):149--181, 2009.

\bibitem{celorrio1999euler}
R.~Celorrio and F.-J. Sayas.
\newblock The {E}uler-{M}aclaurin formula in presence of a logarithmic
  singularity.
\newblock {\em BIT Numerical Mathematics}, 39(4):780--785, 1999.

\bibitem{chawla1974modified}
M.~Chawla and T.~Ramakrishnan.
\newblock {Modified Gauss-Jacobi quadrature formulas for the numerical
  evaluation of Cauchy type singular integrals}.
\newblock {\em BIT Numerical Mathematics}, 14(1):14--21, 1974.

\bibitem{chawla1974numerical}
M.~Chawla and T.~Ramakrishnan.
\newblock {Numerical evaluation of integrals of periodic functions with Cauchy
  and Poisson type kernels}.
\newblock {\em Numerische Mathematik}, 22(4):317--323, 1974.

\bibitem{davis2007methods}
P.~J. Davis and P.~Rabinowitz.
\newblock {\em {Methods of Numerical Integration}}.
\newblock Courier Corporation, 2007.

\bibitem{elliott1979gauss}
D.~Elliott and D.~Paget.
\newblock Gauss type quadrature rules for cauchy principal value integrals.
\newblock {\em Mathematics of Computation}, 33(145):301--309, 1979.

\bibitem{falconer2004fractal}
K.~Falconer.
\newblock {\em Fractal Geometry: Mathematical Foundations and Applications}.
\newblock John Wiley \& Sons, 2004.

\bibitem{faria2020general}
L.~M. Faria, C.~P{\'e}rez-Arancibia, and M.~Bonnet.
\newblock General-purpose kernel regularization of boundary integral equations
  via density interpolation.
\newblock {\em Computer Methods in Applied Mechanics and Engineering},
  378(113703):1--29, 2021.

\bibitem{fornaro1973numerical}
R.~J. Fornaro.
\newblock Numerical evaluation of integrals around simple closed curves.
\newblock {\em SIAM Journal on Numerical Analysis}, 10(4):623--634, 1973.

\bibitem{greenbaum1992numerical}
A.~Greenbaum, L.~Greengard, and A.~Mayo.
\newblock On the numerical solution of the biharmonic equation in the plane.
\newblock {\em Physica D: Nonlinear Phenomena}, 60(1-4):216--225, 1992.

\bibitem{Greengard:2004dg}
L.~Greengard and M.~C. Kropinski.
\newblock {Integral equation methods for Stokes flow in doubly-periodic
  domains}.
\newblock {\em Journal of Engineering Mathematics}, 48(2):157--170, Feb. 2004.

\bibitem{greengard1996integral}
L.~Greengard, M.~C. Kropinski, and A.~Mayo.
\newblock {Integral equation methods for Stokes flow and isotropic elasticity
  in the plane}.
\newblock {\em Journal of Computational Physics}, 125(2):403--414, 1996.

\bibitem{greengard1987fast}
L.~Greengard and V.~Rokhlin.
\newblock A fast algorithm for particle simulations.
\newblock {\em Journal of computational physics}, 73(2):325--348, 1987.

\bibitem{hackbusch1995}
W.~Hackbusch.
\newblock {\em Integral Equations: Theory and Numerical Treatment}.
\newblock Birkh\"auser, 1995.

\bibitem{hackbusch2015hierarchical}
W.~Hackbusch.
\newblock {\em Hierarchical Matrices: Algorithms and Analysis}, volume~49.
\newblock Springer, 2015.

\bibitem{helsing2001complex}
J.~Helsing and A.~Jonsson.
\newblock Complex variable boundary integral equations for perforated infinite
  planes.
\newblock {\em Engineering Analysis with Boundary Elements}, 25(3):191--202,
  2001.

\bibitem{helsing2008evaluation}
J.~Helsing and R.~Ojala.
\newblock On the evaluation of layer potentials close to their sources.
\newblock {\em Journal of Computational Physics}, 227(5):2899--2921, 2008.

\bibitem{henriciVol1}
P.~Henrici.
\newblock {\em Applied and Computational Complex Analysis, Volume 1: Power
  Series, Integration, Conformal Mapping, Location of Zeros}.
\newblock John Wiley \& Sons, 1974.

\bibitem{henriciVol3}
P.~Henrici.
\newblock {\em Applied and Computational Complex Analysis, Volume 3: Discrete
  Fourier Analysis, Cauchy Integrals, Construction of Conformal Maps, Univalent
  Functions}.
\newblock John Wiley \& Sons, 1986.

\bibitem{hunter1972some}
D.~Hunter.
\newblock {Some Gauss-type formulae for the evaluation of Cauchy principal
  values of integrals}.
\newblock {\em Numerische Mathematik}, 19(5):419--424, 1972.

\bibitem{ioakimidis1991numerical}
N.~Ioakimidis, K.~Papadakis, and E.~Perdios.
\newblock Numerical evaluation of analytic functions by {Cauchy's} theorem.
\newblock {\em BIT Numerical Mathematics}, 31(2):276--285, 1991.

\bibitem{jaswon1977integral}
M.~A. Jaswon and G.~T. Symm.
\newblock {\em Integral Equation Methods in Potential Theory and
  Elastostatics}.
\newblock Academic Press, 1977.

\bibitem{javed2014trapezoidal}
M.~Javed and L.~N. Trefethen.
\newblock A trapezoidal rule error bound unifying the {E}uler--{M}aclaurin
  formula and geometric convergence for periodic functions.
\newblock {\em Proceedings of the Royal Society A: Mathematical, Physical and
  Engineering Sciences}, 470(2161):20130571, 2014.

\bibitem{johnson2011notes}
S.~G. Johnson.
\newblock {Notes on FFT-based differentiation}.
\newblock {\em MIT Applied Mathematics}, (April), 2011.

\bibitem{klockner2013quadrature}
A.~Kl{\"o}ckner, A.~Barnett, L.~Greengard, and M.~O'Neil.
\newblock Quadrature by expansion: A new method for the evaluation of layer
  potentials.
\newblock {\em Journal of Computational Physics}, 252:332--349, 2013.

\bibitem{Kolm:2001bt}
P.~Kolm and V.~Rokhlin.
\newblock {Numerical quadratures for singular and hypersingular integrals}.
\newblock {\em Computers and Mathematics with Applications}, 41(3-4):327--352,
  Feb. 2001.

\bibitem{kozlov1997elliptic}
V.~Kozlov, V.~Mazya, and J.~Rossman.
\newblock {\em Elliptic Boundary Value Problems in Domains with Point
  Singularities}, volume~52.
\newblock American Mathematical Soc., 1997.

\bibitem{Kress:1990vm}
R.~Kress.
\newblock {A Nystr{\"o}m method for boundary integral equations in domains with
  corners}.
\newblock {\em Numerische Mathematik}, 58(1):145--161, 1990.

\bibitem{Kress:1995}
R.~Kress.
\newblock {On the numerical solution of a hypersingular integral equation in
  scattering theory}.
\newblock {\em Journal of Computational and Applied Mathematics},
  61(3):345--360, 1995.

\bibitem{kress2012linear}
R.~Kress.
\newblock {\em Linear Integral Equations}, volume~82.
\newblock Springer, 3rd edition, 2014.

\bibitem{kress2014collocation}
R.~Kress et~al.
\newblock A collocation method for a hypersingular boundary integral equation
  via trigonometric differentiation.
\newblock {\em Journal of Integral Equations and Applications}, 26(2):197--213,
  2014.

\bibitem{kropinski1999integral}
M.~Kropinski.
\newblock Integral equation methods for particle simulations in creeping flows.
\newblock {\em Computers and Mathematics with Applications}, 38(5-6):67--87,
  1999.

\bibitem{Kropinski:2002fd}
M.~Kropinski.
\newblock Numerical methods for multiple inviscid interfaces in creeping flows.
\newblock {\em Journal of Computational Physics}, 180(1):1--24, July 2002.

\bibitem{kythe2019handbook}
P.~K. Kythe.
\newblock {\em Handbook of Conformal Mappings and Applications}.
\newblock CRC Press, 2019.

\bibitem{lyness1967numerical}
J.~Lyness and L.~Delves.
\newblock On numerical contour integration round a closed contour.
\newblock {\em Mathematics of Computation}, 21(100):561--577, 1967.

\bibitem{mayo1984fast}
A.~Mayo.
\newblock {The fast solution of Poisson's and the biharmonic equations on
  irregular regions}.
\newblock {\em SIAM Journal on Numerical Analysis}, 21(2):285--299, 1984.

\bibitem{mikhlin1964integral}
S.~G. Mikhlin.
\newblock {\em Integral Equations and Their Applications to Certain Problems in
  Mechanics, Mathematical Physics and Technology}.
\newblock Pergamon Press, 1964.

\bibitem{monegato1982numerical}
G.~Monegato.
\newblock {The numerical evaluation of one-dimensional Cauchy principal value
  integrals}.
\newblock {\em Computing}, 29(4):337--354, 1982.

\bibitem{monegato1994numerical}
G.~Monegato.
\newblock Numerical evaluation of hypersingular integrals.
\newblock {\em Journal of Computational and Applied Mathematics},
  50(1-3):9--31, 1994.

\bibitem{muskhelishvili2008}
N.~I. Muskhelishvili.
\newblock {\em Singular Integral Equations: Boundary Problems and Function
  Theory and Their Application to Mathematical Physics}.
\newblock Dover Publications Inc., 2008.

\bibitem{paget1972algorithm}
D.~Paget and D.~Elliott.
\newblock {An algorithm for the numerical evaluation of certain Cauchy
  principal value integrals}.
\newblock {\em Numerische Mathematik}, 19(5):373--385, 1972.

\bibitem{papamichael2010numerical}
N.~Papamichael and N.~Stylianopoulos.
\newblock {\em Numerical Conformal Mapping: Domain Decomposition and the
  Mapping of Quadrilaterals}.
\newblock World Scientific, 2010.

\bibitem{perez2018plane}
C.~P{\'e}rez-Arancibia.
\newblock A plane-wave singularity subtraction technique for the classical
  {D}irichlet and {N}eumann combined field integral equations.
\newblock {\em Applied Numerical Mathematics}, 123:221--240, 2018.

\bibitem{perez2019harmonic}
C.~P{\'e}rez-Arancibia, L.~M. Faria, and C.~Turc.
\newblock Harmonic density interpolation methods for high-order evaluation of
  laplace layer potentials in {2D} and {3D}.
\newblock {\em Journal of Computational Physics}, 376:411--434, 2019.

\bibitem{perez2019planewave}
C.~P{\'e}rez-Arancibia, C.~Turc, and L.~Faria.
\newblock {Planewave density interpolation methods for 3D Helmholtz boundary
  integral equations}.
\newblock {\em SIAM Journal on Scientific Computing}, 41(4):A2088--A2116, 2019.

\bibitem{perez2020IEEE}
C.~P\'erez-Arancibia, C.~Turc, L.~M. Faria, and C.~Sideris.
\newblock Planewave density interpolation methods for the {EFIE} on simple and
  composite surfaces.
\newblock {\em IEEE Transactions on Antennas and Propagation}, 2020.

\bibitem{saad1986gmres}
Y.~Saad and M.~H. Schultz.
\newblock {GMRES: A generalized minimal residual algorithm for solving
  nonsymmetric linear systems}.
\newblock {\em SIAM Journal on Scientific and Statistical Computing},
  7(3):856--869, 1986.

\bibitem{symm1966integral}
G.~T. Symm.
\newblock An integral equation method in conformal mapping.
\newblock {\em Numerische Mathematik}, 9(3):250--258, 1966.

\bibitem{symm1967numerical}
G.~T. Symm.
\newblock Numerical mapping of exterior domains.
\newblock {\em Numerische Mathematik}, 10(5):437--445, 1967.

\bibitem{theocaris1979method}
P.~S. Theocaris and N.~I. Ioakimidis.
\newblock A method of numerical solution of cauchy-type singular integral
  equations with generalized kernels and arbitrary complex singularities.
\newblock {\em Journal of Computational Physics}, 30(3):309--323, 1979.

\bibitem{trefethen2014exponentially}
L.~N. Trefethen and J.~Weideman.
\newblock The exponentially convergent trapezoidal rule.
\newblock {\em SIAM Review}, 56(3):385--458, 2014.

\bibitem{vetterling1992numerical}
W.~T. Vetterling and W.~H. Press.
\newblock {\em Numerical Recipes in Fortran: The Art of Scientific Computing}.
\newblock Cambridge University Press, 1992.

\bibitem{wala2018conformal}
M.~Wala and A.~Kl{\"o}ckner.
\newblock Conformal mapping via a density correspondence for the double-layer
  potential.
\newblock {\em SIAM Journal on Scientific Computing}, 40(6):A3715--A3732, 2018.

\bibitem{waldvogel2006fast}
J.~Waldvogel.
\newblock {Fast construction of the Fej{\'e}r and Clenshaw--Curtis quadrature
  rules}.
\newblock {\em BIT Numerical Mathematics}, 46(1):195--202, 2006.

\end{thebibliography}

\end{document}